\documentclass{amsart}

\usepackage[T1]{fontenc}
\usepackage{mathptmx}
\usepackage[utf8]{inputenc}
\usepackage{amssymb}

\usepackage{xcolor}
\usepackage{hyperref}
\usepackage{bm}
\usepackage{bbm}
\usepackage{mathtools}
\usepackage{tikz-cd}
\usepackage{thmtools}

\usepackage[makeroom]{cancel}

\allowdisplaybreaks

\newtheorem{theorem}{Theorem}
\numberwithin{theorem}{section}

\newtheorem{corollary}[theorem]{Corollary}
\newtheorem{lemma}[theorem]{Lemma}
\newtheorem{proposition}[theorem]{Proposition}

\newtheorem{question}[theorem]{Question}

\theoremstyle{remark}
\newtheorem{remark}{Remark}
\theoremstyle{definition}
\newtheorem{example}[theorem]{Example}

\usepackage{comment}
\usepackage[all,cmtip]{xy}

\newcommand{\R}{\mathbb{R}}
\newcommand{\Z}{\mathbb{Z}}
\newcommand{\N}{\mathbb{N}}
\newcommand{\acts}{\curvearrowright}
\newcommand{\Isom}{\operatorname{GL}}
\newcommand{\Fix}{\operatorname{Fix}}
\newcommand{\FSet}{\mathcal{F}(G)}
\newcommand{\Part}{\mathcal{P}}
\newcommand{\Maps}{\mathcal{M}}

\newcommand{\kernel}{N}
\newcommand{\NMaps}{\mathcal{N}}

\newcommand{\AdMaps}{\Maps_\oplus}

\newcommand{\Cob}{\mathcal{L}}
\newcommand{\wCob}{{\overline{\Cob}}}
\newcommand{\nmap}{\xi}
\newcommand{\witness}{\mathcal{R}_\varphi}
\newcommand{\Lp}{\mathcal{L}}
\newcommand{\koopman}{\kappa}

\usetikzlibrary{decorations.markings}
\tikzset{
  negate/.style={
    decoration={
      markings,
      mark= at position 0.5 with {
        \node[transform shape] (tempnode) {$/$};
      },
    },
    postaction={decorate},
  },
}

\usepackage{hyperref}
\hypersetup{
	colorlinks=true,
	linktocpage=true,
	linkcolor=[RGB]{161,1,49},
	citecolor=[RGB]{10,136,169},      
	urlcolor=cyan,
}

\newcommand{\vertiii}[1]{{\left\vert\kern-0.25ex\left\vert\kern-0.25ex\left\vert #1 
    \right\vert\kern-0.25ex\right\vert\kern-0.25ex\right\vert}}

\newcommand{\vertG}[1]{{\left\vert\kern-0.25ex\left\vert\kern-0.25ex\left\vert #1 
    \right\vert\kern-0.25ex\right\vert\kern-0.25ex\right\vert}_G}

\newcommand{\vertsup}[1]{{\left\vert\kern-0.25ex\left\vert\kern-0.25ex\left\vert #1 
    \right\vert\kern-0.25ex\right\vert\kern-0.25ex\right\vert}_\infty}

\newcommand{\vertK}[1]{{\left\vert\kern-0.25ex\left\vert\kern-0.25ex\left\vert #1 
    \right\vert\kern-0.25ex\right\vert\kern-0.25ex\right\vert}_\NMaps}

\title[Ergodic theorems for set maps under weak forms of additivity]{Ergodic theorems for set maps\\under weak forms of additivity}

\author{Raimundo Briceño}
\author{Godofredo Iommi}
\address{Facultad de Matem\'aticas, Pontificia Universidad Cat\'olica de Chile. Santiago, Chile}
\email{raimundo.briceno@uc.cl}
\urladdr{http://www.mat.uc.cl/~raimundo.briceno}
\email{godofredo.iommi@gmail.com}
\urladdr{http://www.mat.uc.cl/~giommi}

\subjclass[2020]{Primary 37A15, 46B42, 37A30; Secondary 37B02, 37A60, 37C85.}
\date{}
\keywords{Countable amenable group; non-additive set map; ergodic theorems}

\begin{document}

\begin{abstract}
We investigate various relaxations of additivity for set maps into Banach spaces in the context of representations of amenable groups. Specifically, we establish conditions under which asymptotically additive and almost additive set maps are equivalent. For Banach lattices, we further show that these notions are related to a third weak form of additivity adapted to the order structure of the space. By utilizing these equivalences and reducing non-additive settings to the additive one by finding suitable additive realizations, we derive new non-additive ergodic theorems for amenable group representations into Banach spaces and streamline proofs of existing results in certain cases.
\end{abstract}

\maketitle
\setcounter{tocdepth}{2}
\tableofcontents

\section{Introduction}

A central theme in additive ergodic theory is the study of the asymptotic behavior of Birkhoff sums, which can be viewed as additive sequences of functions. Specifically, given a transformation $T \colon X \to X$ and an observable $f \colon X \to \mathbb{R}$, we define the Birkhoff sums $S_nf(x) = \sum_{i=0}^{n-1} f(T^i x)$. If we take $f_n = S_n f$, these sums satisfy the additive cocycle property $f_{n+m}(x) = f_n(x) + f_m(T^n x)$. However, this classical framework does not address many problems that arise naturally in dynamical systems, probability theory, and statistical mechanics. For example, the study of the growth of norms of matrix products \cite{furstenberg1960products, viana}, which is closely related to Lyapunov exponents and the hyperbolicity of dynamical systems in higher-dimensional phase spaces, falls outside its scope. Similarly, in percolation theory, the objects of interest often lack additive properties \cite{hammersley1965first, kingman1968}.

\emph{Non-additive ergodic theory} generalizes classical ergodic theory to handle these cases. A weak form of additivity is typically required to establish ergodic theorems and develop the theory. One of the earliest such generalizations is \emph{sub-additivity}. A sequence of functions $(f_n)_n$ in $L^1_\mu(X)$, where $T$ preserves a probability measure $\mu$, is sub-additive if for every $n, m \in \mathbb{N}$ and $\mu$-almost every $x \in X$,  
\[
f_{n+m}(x) \leq f_n(x) + f_m(T^n x).
\]  
In 1968, Kingman \cite{kingman1968} proved, under some extra integrability assumptions, a pointwise ergodic theorem for such sequences, providing tools to study maximal Lyapunov exponents as introduced in \cite{furstenberg1960products}. Sub-additive ergodic theory has since been extensively developed \cite{barreira1, cao, falconer1988, kaenmaki}.

Questions of this nature also arise in the context of spatial stochastic processes, where the phase space is $X = \Omega^G$, composed of an index group $G$ and a set of states $\Omega$, which typically carries a measure-theoretical or topological structure. Here, instead of sequences of functions, it is natural to consider \emph{set maps} $\varphi \colon \mathcal{F}(G) \to L^1_\mu(X, \mathbb{R})$, where $\mathcal{F}(G)$ is the collection of non-empty finite subsets of $G$. Such maps, originally used to describe quantities like entropy, volume, or energy, admit a notion of sub-additivity, allowing for the extension of sub-additive ergodic theory to this setting \cite{akcoglu1981, dooley2014sub, nguyen1979ergodic2, smythe1976multiparameter}.

In recent years, other weak notions of additivity have been proposed. For example, Barreira \cite{barreira1} studied sequences of real valued functions for which there exists $C > 0$ such that for every $m,n \in \N$,
$$
\left|f_{n+m}(x) -  f_n(x) - f_{m}(T^n x)\right| \leq C   \qquad   \text{for every } x \in X.
$$
He called this condition ``almost additivity,'' however in order to differentiate it from a more general notion we call it \emph{almost additive with constant error}. 
The study of such sequences is motivated by their relevance in the dimension theory of non-conformal dynamical systems. Their ergodic theory has been systematically developed and is now well-established \cite{bkm, barreira2, barreira1, iommi-yuki, mummert, zhao2011asymptotically}. In the setting of spatial stochastic processes, Pogorzelski \cite{pogorzelski2013almost} extended the concept of almost additivity to set maps, demonstrating its utility in studying percolation models on Cayley graphs of finitely generated amenable groups. Actually, Pogorzelski and  Schwarzenberger \cite[Lemma~8.3]{pogo-sch} proved that the function that counts clusters in a given set, up to a certain size, is almost additive in that sense.

In \cite[\S 3]{bricenoiommi1}, we introduced a complete semi-normed space of set maps, which provides a unified framework for defining additive properties and proving fundamental results (see \S \ref{sec:space-maps} for details). Fix a countable discrete amenable group $G$ and a Banach space $(V, \|\cdot\|)$ (or, more generally, a complete semi-normed space). For a uniformly bounded representation $\pi \colon G \to \Isom(V)$, where $\Isom(V)$ denotes the group of bounded invertible linear operators on $V$, we define the complete semi-normed space $(\Maps, \vertsup{\cdot})$ of bounded $G$-equivariant set maps. The subset of \emph{additive set maps} $\AdMaps$ forms a complete subspace of $\Maps$. Using this framework, we say that a set map $\varphi \in \Maps$ is \emph{almost additive} if there exists a bounded set map $b\colon\FSet\to\R$ with $\vertG{b}=0$ and $b(F)=b(g\cdot F)$ such that
$$
\left\|\varphi(F)-\sum_{E \in \mathcal{P}} \varphi(E)\right\| \leq \sum_{E \in \mathcal{P}} b(E),
$$
for every $F \in \FSet$ and every \emph{monotone invariant} partition $\Part$ of $F$. Here, $\vertG{\cdot}$ is a semi-norm that measures the asymptotical behavior of set maps (see \S \ref{sec:almost-a} for details). When restricted to countable amenable groups, our definition of almost additivity is more general than the one considered by Pogorzelski (see \S \ref{sec:mean-ergodic}), and when $G=\Z$, our notion is more general than the one introduced by Barreira (see \S \ref{sec:aa-constant}). This definition has also a version adapted to Riesz spaces; we develop this idea in \S \ref{sec:aa_riesz} and go on to define the notion of \emph{Riesz-almost additive}. 

Yet another additivity assumption that has been considered was introduced by Feng and Huang \cite{feng-huang}. A sequence of real-valued continuous functions $(f_n)_n$ is \emph{asymptotically additive} if
\begin{equation*} \label{aa-fh}
\inf_{f \in \mathcal{C}(X)} \limsup_{n \to \infty} \frac{1}{n} \left\|  f_n - \sum_{i=0}^{n-1} f \circ T^i  \right\|_{\infty} =0,
\end{equation*}
where $\mathcal{C}(X)$ is the space of continuous functions and $\|\cdot\|_{\infty}$ the uniform norm. Their main motivation stems from the study of Lyapunov exponents of matrix products and the Lyapunov exponents of differential maps on nonconformal repellers. This notion was extended to higher dimensions by Yan \cite{yan2013sub}. In \cite{bricenoiommi1}, we proposed a generalization of asymptotical additivity for set maps $\varphi \in \Maps$:
\begin{equation} \label{defi_aa}
\inf_{\psi \in \AdMaps} \limsup_{F \to G}\left\|\frac{\varphi(F)}{|F|} -\frac{\psi(F)}{|F|}\right\| =0,
\end{equation}
where the limit superior is taken along arbitrary F{\o}lner sequences (see \S \ref{sec:set} for definitions). This notion is broader and more comprehensive than the one proposed by Feng and Huang \cite{feng-huang}. In \cite[Theorem~1]{bricenoiommi1}, we proved that the infimum in Equation \eqref{defi_aa} is actually attained. This result, which generalizes previous work of Cuneo \cite{cuneo2020additive}, establishes that there exists $v \in V$ such that $\varphi(F)/|F|$ can be approximated by the corresponding Birkhoff average of $v$
(see Corollary~\ref{cor:cuneo} or \cite[Corollary~1]{bricenoiommi1}). Such a result allows for the reduction of questions involving asymptotically additive set maps to additive set maps.

Our first main result establishes relationships among these additivity notions, providing a detailed classification.

\begin{restatable}{thm}{theoremone}
\label{thm:main}
Let $G$ be a countable amenable group, $(V,\|\cdot\|)$ a complete semi-normed space, $\pi\colon G \to \Isom(V)$ a uniformly bounded representation, and $\varphi\colon \FSet \to V$ a boun\-ded and $G$-equivariant set map. Then, 
$$
\varphi \text{ is asymptotically additive}  \qquad   \implies   \qquad \varphi \text{ is almost additive},
$$
and, if $G$ is residually finite, then the converse holds. Moreover, if $(V,\|\cdot\|,\ll)$ is a complete semi-normed Riesz space, then
$$
\varphi \text{ is Riesz-almost additive} \qquad \implies  \qquad \varphi \text{ is almost additive}.
$$
The converse does not hold in general, but for the Banach lattice of bounded real-valued functions over a set $X$, it holds if $\pi$ is unital.
\end{restatable}

This result completely clarifies the hierarchy of weak additivity notions. A diagram summarizing these relationships is presented in Figure \ref{diagram}.

Since the inception of non-additive ergodic theory, a common strategy for proving non-additive ergodic theorems has been to reduce them to the additive case. This reduction involves demonstrating that a given non-additive phenomenon can be approximated by an additive one, plus a negligible error term. Kingman’s original proof of the sub-additive ergodic theorem \cite{kingman1968} exemplifies this approach. Similarly, in Theorem~\ref{thm:cuneo} and Corollary~\ref{cor:cuneo} (originally proved in \cite[Theorem~1]{bricenoiommi1}), we show that an asymptotically additive set map taking values in a complete semi-normed space equals an additive one, plus a negligible error term. Specifically, there exists $v \in V$ such that
$$
\lim_{F \to G} \left\|\frac{\varphi(F)}{|F|}-A_F v\right\| = 0, \qquad \text{where} \qquad A_F v= \frac{1}{|F|}\sum_{g \in F} \pi(g^{-1})v \qquad   \text{for }  v \in V.
$$
That is, we can approximate $\varphi(F)$ with the Birkhoff sum of $v$. Together with Theorem~\ref{thm:main}, this framework enables us to derive new ergodic theorems for non-additive set maps into Banach spaces in the context of representations of amenable groups. Moreover, it provides a clearer and simpler perspective on the proofs of known results in the area.

Our first result is a generalization of the mean ergodic theorem, originally proved by von Neumann \cite{vn}, and later extended to encompass more general dynamical systems, groups, and maps. Relevant contributions in this context include \cite[Theorem
~4.22]{kerr2016ergodic}, \cite[\S 6.4.1]{krengel1985}, and \cite[Theorem~4.4]{pogorzelski2013almost}. Let $\Fix(\pi)$ be the subspace of $G$-invariant vectors.

\begin{restatable}{thm}{theoremtwo}
\label{thm:meanasymp}
Let $G$ be a countable discrete amenable group, $(V,\|\cdot\|)$ a Banach space, and $\pi\colon G \to \Isom(V)$ a uniformly bounded representation. Suppose that, for each $v \in V$, the convex hull $\operatorname{co}\left\{\pi(g) v : g \in G\right\}$ is relatively weakly compact. Then, there is a continuous projection $P$ on $V$ such that for every bounded and $G$-equivariant asymptotically additive set map $\varphi\colon \FSet \to V$, it holds that
$$
\lim_{F \to G}\left\|\frac{\varphi(F)}{|F|} - P v\right\|=0
$$
for every additive realization $v \in V$ of $\varphi$. Moreover, if $(V,\|\cdot\|)$ is a Hilbert space, the assumption on the convex hull is not necessary. If, in addition, $\pi$ is isometric, then $P$ is the
orthogonal projection of $V$ onto $\Fix(\pi)$.
\end{restatable}

By Theorem~\ref{thm:main}, this result also applies to almost additive set maps when $G$ is residually finite.

Next, we address the pointwise ergodic theorem, initially proved by Birkhoff \cite{birkhoff} and extended by Lindenstrauss \cite[Theorem~2.1]{lindenstrauss2001pointwise} to amenable groups. Extensions to set maps taking values in Banach spaces have been developed by Pogorzelski \cite[Theorem~6.8]{pogorzelski2013almost}. Here, we prove a pointwise ergodic theorem for asymptotically additive set maps with values in a Banach space. For $1 \leq p < \infty$, a measurable space $(\Omega,\mathcal{B},\mu)$, and a Banach space $(Y,\|\cdot\|_Y)$, consider the intersection of Bochner spaces $L^p_\infty(\Omega, Y) := L^p(\Omega, Y) \cap L^\infty(\Omega, Y)$, which is a Banach space with a natural norm (see \S \ref{sec:point}).

\begin{restatable}{thm}{theoremthree}
\label{thm:pointwise-asymp}
Let $G$ be a countable discrete amenable group acting on a $\sigma$-finite measure space $(\Omega, \mathcal{B}, \mu)$ by measure-preserving transformations, $(Y,\|\cdot\|_Y)$ a reflexive Banach space, and $\pi\colon G \to \Isom(L^p_\infty(\Omega, Y))$ a uniformly bounded representation for some $1 \leq p < \infty$. Assume that there is some group homomorphism $\theta\colon G \to G$ and a constant $B>0$ such that for every $g \in G$ and each $f \in L^p_\infty(\Omega, Y)$,
$$
\left\|(\pi(g)f)(\omega)\right\|_Y \leq B\left\|f\left((\theta(g))^{-1} \cdot \omega\right)\right\|_Y  \qquad   \mu(\omega)\text{-a.e.}
$$
Then, if $(F_n)_n$ is a tempered F{\o}lner sequence in $G$, for each bounded and $G$-equivariant asymptotically additive set map $\varphi\colon \FSet \to L^p_\infty(\Omega, Y)$, there exists $\overline{f} \in L^p_\infty(\Omega, Y)$ such that
$$
\lim_{n \to \infty}\left\|\frac{\varphi(F_n)}{|F_n|}(\omega) - \overline{f}(\omega)\right\|_Y=0    \qquad   \mu(\omega)\text{-a.e.}
$$
Moreover, $\pi(g) \overline{f}=\overline{f}$ in $L^p_\infty(\Omega, Y)$ for all $g \in G$.
\end{restatable}
As with Theorem~\ref{thm:meanasymp}, this result extends to almost additive set maps when $G$ is residually finite.

In conclusion, our results clarify the hierarchy of weak additivity notions and demonstrate their utility in reducing non-additive problems to additive ones, enabling the derivation of ergodic theorems in broader settings.

\section{Preliminaries}

\subsection{Amenable groups}

Let $G$ be a countable discrete group with identity element $1_G$. We denote by $\mathcal{F}(G)$ the set of nonempty finite subsets of $G$. Given $K \in \mathcal{F}(G)$ and $\delta > 0$, we say that $F \in \mathcal{F}(G)$ is {\bf $(K,\delta)$-invariant} if
$$
|KF \Delta F| \leq \delta|F|,
$$
and define its {\bf $K$-interior} as 
$$
\operatorname{Int}_K(F) = \{g \in G: Kg \subseteq F\},
$$
and its {\bf $K$-boundary} as
$$
\partial_K(F) = \{g \in G: Kg \cap F \neq \emptyset, Kg \cap (G \setminus F) \neq \emptyset\}.
$$

We say that $G$ is {\bf amenable} if for every $(K,\delta) \in \mathcal{F}(G) \times \R_{> 0}$ there exists a $(K,\delta)$-invariant set $F \in \mathcal{F}(G)$. A sequence $(F_n)_n$ in $\mathcal{F}(G)$ is {\bf (left) F{\o}lner} for $G$ if
$$
\lim_{n \to \infty} \frac{\left|F_n \Delta gF_n\right|}{\left|F_n\right|}=0 \qquad \text{for any } g \in G.
$$
A countable discrete group is amenable if and only if it has a F{\o}lner sequence. A F{\o}lner sequence $(F_n)_n$ is said to be {\bf tempered} if there exists a constant $D > 0$ such that
$$
\left|\left(\bigcup_{i<n} F_i^{-1}\right) F_n\right| \leq D\left|F_n\right| \qquad   \text{ for every } n.
$$
It can be checked that every F{\o}lner sequence has a tempered F{\o}lner subsequence.

\subsection{Semi-normed and Riesz spaces}

A semi-normed vector space $(V,\|\cdot\|)$ is said to be {\bf complete} if for any Cauchy sequence $(v_n)_n$ in $V$, there exists $v \in V$ such that $\|v_n - v\| \to 0$. Throughout, we assume that the seminorm on $V$
is not identically zero. The {\bf kernel} of $\|\cdot\|$ is the set $\kernel := \{v \in V: \|v\| = 0\}$. It is known that the set $\kernel$ is a closed and complete subspace of $V$.

A {\bf Riesz space} is a real partially ordered vector space $(V, \ll)$ such that, for every $u,v,w \in V$ and $\alpha \in \R$,
\begin{enumerate}
\item[(a)] if $u \ll v$, then $u+w \ll v+w$;
\item[(b)] if $0 \ll v$ and $0 \leq \alpha$, then $0 \ll \alpha v$;
\item[(c)] there exists a least upper bound $u \lor v \in V$ for $u$ and $v$.
\end{enumerate}
The previous assumptions imply that, for any $u,v \in V$, there is also a greatest lower bound $u \land v \in V$  given by $-((-u) \lor (-v))$. We define $|v| = v \lor (-v)$. Some useful properties that any Riesz space satisfies are the following (see \cite[Ch. 2, \S 11-12]{luxemburg1}):
\begin{enumerate}
\item $2(u \lor v) = u+v + |u-v|$.
\item $||u|-|v|| \ll |u+v| \ll |u|+|v|$.
\item $0 \ll |u \lor v| \ll |u| \lor |v| \ll |u|+|v|$.
\end{enumerate}

A consequence of (1) is that, in order to check condition (c) above, it suffices to prove that $|v|$ exists for every $v \in V$. A vector $v \in V$ is called {\bf positive} whenever $0 \ll v $. Observe that $0 \ll |v|$ for all $v \in V$. The set $V^+ = \{v \in V: 0 \ll v\}$ is called the {\bf positive cone} of $V$. 

A semi-norm $\|\cdot\|$ on a Riesz space $(V, \ll)$ is a {\bf Riesz semi-norm} if
$$
|u| \ll |v| \qquad \implies \qquad \|u\| \leq \|v\|.
$$
In this case, $\||v|\| = \|v\|$ for every $v \in V$.

A tuple $(V,\|\cdot\|,\ll)$ such that $(V,\ll)$ is a Riesz space and $\|\cdot\|$ is a Riesz semi-norm on it will be called {\bf semi-normed Riesz space}. A semi-normed Riesz space is said to be a {\bf complete semi-normed Riesz space} if the associated semi-normed vector space is complete. A complete semi-normed Riesz space such that $\|\cdot\|$ is a norm is called {\bf Banach lattice}. Notice that $(\R,|\cdot|,\leq)$ is a Banach lattice for $|\cdot|$ the absolute value and $\leq$ the usual order.

\subsection{Group representations}
\label{sec:1-3}

Let $(V,\|\cdot\|)$ be a semi-normed vector space, and denote by $\Isom(V)$ the group of bounded invertible linear operators on $V$. For a countable discrete group $G$, a group homomorphism $\pi\colon G \to \Isom(V)$ is called a {\bf representation}, and $\pi$ is said to be {\bf uniformly bounded} if, for $\left\|\cdot\right\|_{\operatorname{op}}$ the linear operator norm,
$$
C_\pi := \sup_{g \in G}\left\|\pi(g)\right\|_{\operatorname{op}}<\infty.
$$
If $C_\pi = 1$, we say $\pi$ is {\bf isometric}. Given a representation $\pi$, define the {\bf ergodic sum} and {\bf ergodic average}, respectively, by
$$
S_F v = \sum_{g \in F} \pi(g^{-1}) v \qquad \text{and} \qquad A_F v = \frac{1}{|F|} S_F v,
$$
for $F \in \FSet$ and $v \in V$. Let $\Fix(\pi) = \{v \in V: \pi(g)v = v \text{ for all } g \in G\}$ be the subspace of {\bf $G$-invariant vectors}. Observe that $v \in \Fix(\pi)$ if and only if $A_F v = v$ for every $F \in \FSet$.

If $(V,\ll)$ is a Riesz space, a linear bijection $T\colon V\to V$ is a {\bf lattice automorphism} if
$$
T(u\lor v)=Tu\lor Tv \qquad \text{for all }u,v\in V.
$$
A representation $\pi$ {\bf acts by lattice automorphisms} if $\pi(g)$ is a lattice automorphism for every $g\in G$.

\subsection{Bochner spaces}

Consider a $\sigma$-finite measure space $(\Omega, \mathcal{B}, \mu)$ and a Banach space $(Y,\|\cdot\|_Y)$. A function $s\colon \Omega \to Y$ is a {\bf step function} if it can be written as $s = \sum_{i=1}^k y_i\mathbbm{1}_{A_i}$ for some $k \in \N$, and $y_i \in Y$ and $A_i \in \mathcal{B}$ with $\mu(A_i) < \infty$ for every $i=1,\dots,k$. A function $f \colon \Omega \to Y$ is {\bf Bochner-measurable} if there exists a sequence $(s_n)_n$ of step functions $s_n\colon \Omega \to Y$ such that
$$
f(\omega)=\lim_{n \to \infty} s_n(\omega) \qquad   \mu(\omega)\text{-a.e.}
$$

Note that, if $f$ is Bochner-measurable, then $\omega \mapsto \|f(\omega)\|_Y$ is Lebesgue-measurable. For $1 \leq p \leq \infty$, let $(\Lp^p(\Omega, Y),\|\cdot\|_{\Lp^p(\Omega,Y)})$ be the complete semi-normed space of Bochner-measurable functions for which $\|f\|_{\Lp^p(\Omega,Y)}  < \infty$, where
$$
\|f\|_{\Lp^p(\Omega,Y)} =   \begin{cases}
    \left(\int_\Omega \|f(\omega)\|^p_Y d\mu(\omega)\right)^{1/p} &   \text{ if } 1 \leq p < \infty,  \\
    \operatorname{ess} \sup_{\omega \in \Omega} \|f(\omega)\|_Y &   \text{ if } p = \infty.
\end{cases} 
$$
For $1\leq p < \infty$, a Bochner-measurable function $f$ belongs to $\Lp^p(\Omega,Y)$ if and only if there exists a sequence of step functions $(s_n)_n$ such that $\|f(\cdot)-s_n(\cdot)\|_Y$ converges to $0$ in the Lebesgue space $\Lp^p(\Omega,\R)$ as $n$ goes to infinity. If $f \in \Lp^1(\Omega,Y)$, the {\bf Bochner integral} is given by
$$
\int_\Omega f d\mu = \lim_{n \to \infty} \int_\Omega s_n d \mu.
$$
This limit always exists and is independent of the approximating sequence of step functions $(s_n)_n$. For more details, see \cite[Ch. 6, \S 31]{zaanen1}. 

The {\bf Bochner space} $(L^p(\Omega, Y),\|\cdot\|_{L^p(\Omega,Y)})$ is the Banach space obtained as the quotient of $(\Lp^p(\Omega, Y),\|\cdot\|_{\Lp^p(\Omega,Y)})$ with its corresponding kernel. If we further assume some order structure on $Y$, the spaces $\Lp^p(\Omega,Y)$ and $L^p(\Omega,Y)$ inherit it, that is, if $Y$ is a Banach lattice, then $L^p(\Omega, Y)$ is a Banach lattice as well.

\begin{proposition}
\label{prop:Lp}
Let $1 \leq p \leq \infty$. If $(Y,\preccurlyeq)$ is a Banach lattice, then $\Lp^p(\Omega,Y)$ is a complete semi-normed Riesz space with the pointwise partial order $\ll$ given by
$$
f \ll h \qquad \iff \qquad f(\omega) \preccurlyeq h(\omega) \qquad  \text{for every $\omega \in \Omega$},
$$
and $L^p(\Omega,Y)$ is a Banach lattice with the partial order induced by the quotient.
\end{proposition}

\begin{proof}
To see that $(\Lp^p(\Omega,Y),\ll)$ is a Riesz space, we only need to check condition (c) from the definition because (a) and (b) are direct. Then, it suffices to verify that if $f$ belongs to $\Lp^p(\Omega,Y)$, we have that $|f| \in \Lp^p(\Omega,Y)$, where $|f|(\omega) = |f(\omega)|$ in $Y$.

First, let's check that $|f|$ is Bochner-measurable. Since $f$ is Bochner-measurable, there exists a sequence $(s_n)_n$ of step functions $s_n\colon \Omega \to Y$ such that
$$
\lim_{n \to \infty} \|f(\omega) - s_n(\omega)\|_Y = 0 \qquad   \mu(\omega)\text{-a.e.}
$$
Since $\|\cdot\|_Y$ is a Riesz norm, by the triangle inequality,
$$
||f(\omega)|-|s_n(\omega)|| \preccurlyeq |f(\omega)-s_n(\omega)| \implies \||f(\omega)|-|s_n(\omega)|\|_Y \leq \|f(\omega)-s_n(\omega)\|_Y
$$
for every $\omega \in \Omega$. Therefore, noticing that $|s_n|$ is a step function as well, we have that
$$
\lim_{n \to \infty} \||f|(\omega) - |s_n|(\omega)\|_Y \leq \lim_{n \to \infty} \|f(\omega) - s_n(\omega)\|_Y = 0 \qquad   \mu(\omega)\text{-a.e.},
$$
so $|f|$ is Bochner-measurable. Next, observing that $\||f(\omega)|\|_Y = \|f(\omega)\|_Y$ for every $\omega \in \Omega$, we obtain that $|f|$ belongs to $\Lp^p(\Omega,Y)$.

In order to see that $(\Lp^p(\Omega,Y),\|\cdot\|_{\Lp^p(\Omega,Y)},\ll)$ is a complete semi-normed Riesz space, we only need to verify that $\|\cdot\|_{\Lp^p(\Omega,Y)}$ is a Riesz semi-norm. Indeed, for $f,h \in \Lp^p(\Omega,Y)$,
\begin{align*}
|f| \ll |h| &   \iff |f(\omega)| = |f|(\omega) \preccurlyeq |h|(\omega) = |h(\omega)| \qquad \text{for every $\omega \in \Omega$} \\
            &   \implies \||f|(\omega)\|_Y= \|f(\omega)\|_Y \leq \|h(\omega)\|_Y = \||h|(\omega)\|_Y \qquad \mu(\omega)\text{-a.e.} \\
            &   \implies    \|f\|_{\Lp^p(\Omega,Y)} \leq \|h\|_{\Lp^p(\Omega,Y)}.
\end{align*}

Finally, to obtain that $(L^p(\Omega,Y),\|\cdot\|_{L^p(\Omega,Y)},\ll)$ is a Banach lattice, we appeal to \cite[Ch. 3, \S 18]{luxemburg1}.
\end{proof}

\subsection{Actions by measure-preserving transformations}

Given a $\sigma$-finite measure space $(\Omega, \mathcal{B}, \mu)$, we say that a countable discrete group $G$ {\bf acts by measure-preserving transformations}, denoted $G \acts (\Omega, \mu)$, if there is a measurable map $\alpha\colon G \times \Omega \to \Omega$ such that, $\mu(\omega)\text{-a.e.}$,
$$
1_G \cdot \omega = \omega \qquad  \text{and} \qquad g \cdot (h \cdot \omega) = (gh) \cdot \omega \qquad \text{for all } g, h \in G,
$$
and each $g$ preserves $\mu$, i.e.,
$$
\mu(g \cdot A)=\mu(A) \qquad \text{for all } g \in G \text{ and } A \in \mathcal{B},
$$
where $g \cdot x$ denotes $\alpha(g,x)$ and $g \cdot A = \{g \cdot \omega: \omega \in A\}$. If $G \acts (\Omega, \mu)$, we will be particularly interested in the {\bf Koopman representation} $\koopman\colon G \to \Isom(L^p(\Omega, Y))$ for $1 \leq p \leq \infty$, given by
$$
\koopman(g) f (\omega) = f(g^{-1} \cdot \omega)  \qquad   \text{for } g \in G, \omega \in \Omega, \text{ and } f \in L^p(\Omega, Y).
$$
Notice that $\koopman$ is isometric, i.e., $C_\koopman = 1$. Indeed, given any $g \in G$ and $f \in L^p(\Omega, Y)$, if $1 \leq p < \infty$, we have
$$
\|\koopman(g) f\|_{L^p(\Omega, Y)}^p = \int_\Omega\|f(g^{-1} \cdot \omega)\|_Y^p d\mu(\omega) = \int_\Omega\|f(\omega)\|_Y^p d\mu(\omega) = \|f\|_{L^p(\Omega, Y)}^p,
$$
because $G$ acts by measure-preserving transformations, and if $p = \infty$, we have
$$
\|\koopman(g) f\|_{L^\infty(\Omega, Y)} = \operatorname{ess}\sup_{\omega \in \Omega} \|f(g^{-1} \cdot \omega)\|_Y = \operatorname{ess}\sup_{\omega \in g \cdot \Omega} \|f(\omega)\|_Y = \|f\|_{L^\infty(\Omega, Y)},
$$
because $G$ acts by almost surely invertible transformations.

\subsection{Set maps} \label{sec:set}

For a countable discrete amenable group $G$ and a semi-normed vector space $(V,\|\cdot\|)$, a function of the form $\varphi\colon \FSet \to V$ will be called a {\bf set map}. Given $L \in V$, we write $\lim_{F \to G} \varphi(F) = L$ if for every $\epsilon > 0$, there exists $(K,\delta) \in \FSet \times \R_{>0}$ such that for every $(K, \delta)$-invariant set $F \in \FSet$ we have that $\|\varphi(F)-L\|<\epsilon$. In this case, we say that {\bf $\varphi(F)$ converges to $L$ as $F$ becomes more and more invariant}. The following result is a vector-valued version of \cite[Lemma~2.1]{bricenoiommi1}.

\begin{lemma}
\label{lem:limFolner}
Let $G$ be a countable discrete amenable group. Given a set map $\varphi\colon \FSet \to V$ and $L \in V$, the following are equivalent.
\begin{enumerate}
    \item[(i)] $\lim_{F \to G} \varphi(F) = L$.
    \item[(ii)] $\lim_{n \to \infty} \varphi(F_n) = L$ for every F{\o}lner sequence $(F_n)_n$.
\end{enumerate}
\end{lemma}

\begin{proof}
Pick a F{\o}lner sequence $(F_n)_n$ and $\epsilon > 0$. If $\lim_{F \to G} \varphi(F) = L$, then there exists $(K,\delta) \in \FSet \times \R_{> 0}$ such that $\|\varphi(F)-L\| < \epsilon$ for every $(K,\delta)$-invariant set $F$. Let $n_0$ be such that $F_n$ is $(K,\delta)$-invariant for every $n \geq n_0$. Then, $\|\varphi(F_n)-L\| < \epsilon$ for every $n \geq n_0$, so $\lim_{n \to \infty} \varphi(F_n) = L$ and we have that $\text{(i)} \implies \text{(ii)}$.

To prove that $\text{(ii)} \implies \text{(i)}$, suppose that $\lim_{n \to \infty} \varphi(F_n) = L$ for every F{\o}lner sequence $(F_n)_n$ and, by contradiction, that there exists $\epsilon > 0$ such that for every $(K,\delta) \in \FSet \times \R_{> 0}$ there exists a $(K,\delta)$-invariant set $F$ such that $\|\varphi(F)-L\| \geq \epsilon$. Pick a cofinal sequence $(K_n,\delta_n)$ and, for each $n$, pick a $(K_n,\delta_n)$-invariant set $F_n$ such that $\|\varphi(F_n)-L\| \geq \epsilon$. Since $(F_n)_n$ is a F{\o}lner sequence, we conclude that $\lim_{n \to \infty} \varphi(F_n)$ is not $L$, which is a contradiction. Therefore, $\lim_{F \to G} \varphi(F) = L$.
\end{proof}

For a real-valued set map $\varphi\colon \FSet \to \R$, we also define its {\bf limit superior} as
$$
\limsup_{F \to G} \varphi(F) := \inf_{(K,\delta)} \sup\{\varphi(F): F \text{ is } (K,\delta)\text{-invariant}\}.
$$

\section{Space of set maps and weak forms of additivity} \label{sec:space-maps}

Throughout this paper, $G$ will be a countable discrete amenable group, $(V,\|\cdot\|)$ a complete semi-normed vector space, $(V,\|\cdot\|,\ll)$ a complete semi-normed Riesz space, and $\pi\colon G \to \Isom(V)$ a uniformly bounded representation of $G$. We stress that a complete semi-normed Riesz space $(V,\|\cdot\|,\ll)$ can be always thought of just a complete semi-normed vector space $(V,\|\cdot\|)$.

In this setting, there are natural left actions $G \acts V$ and $G \acts \FSet$ given by $g \cdot v = \pi(g)v$ and $g \cdot F = Fg^{-1}$, respectively. We say that a set map $\varphi\colon \FSet \to V$ is {\bf $G$-equivariant} if
$$
g \cdot \varphi(F) = \varphi(g \cdot F) \qquad \text{for every $g \in G$ and $F \in \FSet$}.
$$
We also define the semi-norms
$$
\vertsup{\varphi} := \sup_{F \in \FSet}\left\|\frac{\varphi(F)}{|F|}\right\| \qquad \text{and} \qquad 
\vertG{\varphi} := \limsup_{F \to G}\left\|\frac{\varphi(F)}{|F|}\right\|,
$$
and say that $\varphi$ is {\bf bounded} if $\vertsup{\varphi} < \infty$. Consider the {\bf space of set maps}
$$
\Maps := \{\varphi\colon \FSet \to V \mid \varphi \text{ is bounded and $G$-equivariant}\}.
$$
Note that $\|\varphi(g \cdot F)\| \leq \vertsup{\varphi}|F|$ for every $\varphi \in \Maps$, $g \in G$, and $F \in \FSet$. We will be particularly interested in the set
$$
\NMaps = \{\varphi \in \Maps: \vertG{\varphi} = 0\},
$$
i.e., the kernel of $\vertG{\cdot}$, that will be called the {\bf asymptotical kernel (of $\pi$)}. The set $\NMaps$ is a $\vertsup{\cdot}$-closed subspace of $\Maps$ (see \cite[Remark~2]{bricenoiommi1}).

\subsection{Order and completeness}

If $(V,\ll)$ is a Riesz space, there is a natural way to induce a partial order $\lll$ in $\Maps$ by declaring, for any $\varphi,\psi \in \Maps$,
$$
\varphi \lll \psi \qquad \iff \qquad \varphi(F) \ll \psi(F) \qquad \text{for every } F \in \FSet.
$$

\begin{proposition}
\label{prop:riesz}
If $(V,\ll)$ is a Riesz space and $\pi$ acts by lattice automorphisms, then both $(\Maps,\lll)$ and $(\NMaps,\lll)$ are Riesz spaces.
\end{proposition}

\begin{proof}

Let $\varphi,\psi,\chi \in \Maps$. It is routine to check that (a) if $\varphi \lll \psi$, then $\varphi + \chi \lll \psi + \chi$, (b) if $0 \lll \varphi$ and $0 \leq \alpha$, then $0 \lll \alpha\varphi$, and (c) if we define $\varphi \lor \psi$ as the set map given by $(\varphi \lor \psi)(F) = \varphi(F) \lor \psi(F)$ for $F \in \FSet$, we have
$$
\vertsup{\varphi\lor\psi}
\leq \vertsup{\varphi}+\vertsup{\psi}<\infty.
$$
Since $\pi$ acts by lattice automorphisms and $\varphi,\psi$
are $G$-equivariant,
$$
g\cdot(\varphi\lor\psi)(F)
=(g\cdot\varphi(F))\lor(g\cdot\psi(F))
=(\varphi\lor\psi)(g\cdot F).
$$
Thus $\varphi\lor\psi\in\Maps$. Moreover,
\begin{itemize}
\item $\varphi \lll \varphi \lor \psi$, since $\varphi(F) \ll \varphi(F) \lor \psi(F)$ for all $F \in \FSet$, so $\varphi \lor \psi$ is an upper bound for $\varphi$ and, by symmetry, the same holds for $\psi$;
\item if $\varphi,\psi \lll \chi$, then
$\varphi(F),\psi(F) \ll \chi(F)$ for all $F \in \FSet$.
Then, $\varphi(F) \lor \psi(F) \ll \chi(F)$ for all $F$, so $(\varphi \lor \psi)(F) \ll \chi(F)$ for all $F$, and we conclude that $\varphi \lor \psi \lll \chi$, i.e., $\varphi \lor \psi$ is the least upper bound.
\end{itemize}
Therefore, $(\Maps,\lll)$ is a Riesz space. To see that the subspace $\NMaps$ induces a Riesz space $(\NMaps,\lll)$, we only need to check that if $\varphi,\psi \in \NMaps$, then $\varphi \lor \psi \in \NMaps$. Indeed, since
$$
|(\varphi \lor \psi)(F)| = |\varphi(F) \lor \psi(F)| \ll |\varphi(F)| \lor |\psi(F)| \ll |\varphi(F)| + |\psi(F)|
$$
for every $F \in \FSet$, it follows that
$$
\left\|(\varphi \lor \psi)(F)\right\| = \left\||(\varphi \lor \psi)(F)|\right\| \leq \||\varphi(F)| + |\psi(F)|\| \leq \|\varphi(F)\| + \|\psi(F)\|,
$$
so
$$
\vertG{\varphi \lor \psi} \leq \vertG{\varphi} + \vertG{\psi} = 0 + 0 = 0,
$$
and we conclude.
\end{proof}

In \cite{bricenoiommi1}, it was proven that if $(V,\|\cdot\|)$ is a Banach space (resp. complete semi-normed space), then $(\Maps,\vertsup{\cdot})$ is a Banach space (resp. complete semi-normed space) \cite[Proposition~3.1]{bricenoiommi1} and $(\Maps,\vertG{\cdot})$ is a complete semi-normed space \cite[Proposition~4.3]{bricenoiommi1}. Here we prove a lattice version of this result.

\begin{proposition}
\label{prop:riesz-space}
If $(V,\|\cdot\|,\ll)$ is a complete semi-normed Riesz
space and $\pi$ acts by lattice automorphisms, then we have that both $(\Maps,\vertsup{\cdot},\lll)$ and $(\Maps,\vertG{\cdot},\lll)$ are complete semi-normed Riesz spaces. Moreover, if $(V,\|\cdot\|,\ll)$ is a Banach lattice, then $(\Maps,\vertsup{\cdot},\lll)$ is a Banach lattice as well.
\end{proposition}

\begin{proof}
By Proposition~\ref{prop:riesz}, $(\Maps,\lll)$ and $(\NMaps,\lll)$ are Riesz spaces. By \cite[Proposition~3.1]{bricenoiommi1}, $(\Maps,\vertsup{\cdot})$ is a complete semi-normed space (and a Banach space if $(V,\|\cdot\|)$ is a Banach space) and by \cite[Proposition~4.3]{bricenoiommi1}, $(\Maps,\vertG{\cdot})$ is a complete semi-normed space. Therefore, it suffices to check that both $\vertsup{\cdot}$ and $\vertG{\cdot}$ are Riesz semi-norms. Indeed, if $|\varphi| \lll |\psi|$, then $|\varphi(F)| \ll |\psi(F)|$ for every $F \in \FSet$, so $\left\|\varphi(F)\right\| \leq \left\|\psi(F)\right\|$ for every $F$, as well. Then,
$$
\vertsup{\varphi} = \sup_{F \in \FSet} \left\|\frac{\varphi(F)}{|F|}\right\| \leq  \sup_{F \in \FSet} \left\|\frac{\psi(F)}{|F|}\right\|  =  \vertsup{\psi},
$$
and
$$
\vertG{\varphi} = \limsup_{F \to G} \left\|\frac{\varphi(F)}{|F|}\right\| \leq   \limsup_{F \to G} \left\|\frac{\psi(F)}{|F|}\right\|  =  \vertG{\psi},
$$
as desired.
\end{proof}

We will be interested in the asymptotical kernel $\NMaps$ and its positive cone $\NMaps^+$, which will play a key role in this work.

\begin{proposition}
If $(V,\|\cdot\|,\ll)$ is a complete semi-normed Riesz space and $\pi$ acts by lattice automorphisms, then $(\NMaps,\lll)$ is a Riesz subspace of $(\Maps,\lll)$.
\end{proposition}

\begin{proof}
We only need to check that if $\varphi,\psi \in \NMaps$, then $\varphi \lor \psi \in \NMaps$. Indeed,
$$
\vertG{\varphi \lor \psi} \leq \vertG{\varphi} + \vertG{\psi} = 0 + 0 = 0.
$$
\end{proof}

\subsection{Additive set maps}

A set map $\varphi \in \Maps$ will be called {\bf additive} if
$$
\varphi(E \sqcup F) = \varphi(E) + \varphi(F)
$$
for every pair of disjoint sets $E,F \in \FSet$ or, equivalently, if
$$
\varphi(F) = \sum_{E \in \Part} \varphi(E)
$$
for every $F \in \FSet$ and for every partition $\Part$ of $F$. We will denote by $\AdMaps$ the set of additive set maps in $\Maps$. The set $\AdMaps$ is a $\vertsup{\cdot}$-complete subspace of $\Maps$ \cite[Proposition~4.1]{bricenoiommi1}.

\subsection{Asymptotically additive set maps and additive realizations}

A set map $\varphi \in \Maps$ is {\bf asymptotically additive} if
$$
\inf_{\psi \in \AdMaps} \vertG{\varphi - \psi} = 0,
$$
i.e., if it belongs to the closure of $\AdMaps$ with respect to the semi-norm $\vertG{\cdot}$.

\subsection{Almost additive set maps} \label{sec:almost-a}

Let $(\R,|\cdot|,\leq)$ be the Banach lattice of real numbers and consider the trivial representation $\pi_\R\colon G \to \Isom(\R)$ given by $\pi_\R(g) = \operatorname{id}_\R$. Let $\Maps_\R$ be the space of set maps induced by $\pi_\R$, namely
$$
\Maps_\R = \{b\colon \FSet \to \R \mid b \text{ is bounded and $G$-equivariant}\}.
$$
Since $\pi_\R$ is trivial, here $G$-equivariance translates into $G$-invariance, that is,
$$b(F) = \pi_\R(g) b(F) = b(g \cdot F) \qquad \text{for every } F \in \FSet.
$$
We denote by $\NMaps_\R$ the corresponding asymptotical kernel subspace. 

We say that a partition $\Part$ of a set $F \in \FSet$ is \textbf{monotone invariant} if for every $(K,\delta) \in \FSet \times \R_{> 0}$ and every $E \in \Part$,
$$
E \text{ is } (K,\delta)\text{-invariant} \qquad \implies \qquad F \text{ is } (K,\delta)\text{-invariant}.
$$

Given $b \in \NMaps_\R$, we say that a set map $\varphi \in \Maps$ is {\bf $b$-almost additive} if
$$
\left\|\varphi(F)-\sum_{E \in \mathcal{P}} \varphi(E)\right\| \leq \sum_{E \in \mathcal{P}} b(E),
$$
for every $F \in \FSet$ and every monotone invariant partition $\Part$ of $F$. Since $\NMaps_\R$ is a Riesz space, without loss of generality we can assume that $b \in \NMaps^+_\R$, as $b \leq |b|$. If $\varphi \in \Maps$ is $b$-almost additive for some $b \in \NMaps_\R$, we say that $\varphi$ is {\bf almost additive}.

\subsection{Riesz-almost additive set maps} \label{sec:aa_riesz}

Let $(V,\|\cdot\|,\ll)$ be a complete semi-normed Riesz space.
Given $\nmap\in\NMaps$, we say that a set map
$\varphi\in\Maps$ is {\bf $\nmap$-Riesz-almost additive} if
$$
\left|\varphi(F) - \sum_{E \in \Part} \varphi(E)\right| \ll \sum_{E \in \Part} \nmap(E),
$$
for every $F \in \FSet$ and every monotone invariant partition $\Part$ of $F$. Taking the trivial partition $\Part=\{F\}$ in the above inequality gives $0\ll\nmap(F)$ for every $F\in\FSet$. Hence $\nmap\in\NMaps^+$. If $\varphi \in \Maps$ is $\nmap$-Riesz-almost additive for some $\nmap \in \NMaps$, we say that $\varphi$ is {\bf Riesz-almost additive}. Notice that, if $\nmap\in\NMaps$, the set map $b\colon\FSet\to\R$ defined by
$$
b(F)=\sup_{g\in G}\|\nmap(g\cdot F)\|
$$
belongs to $\NMaps_\R^+$, since it is $G$-invariant and $\|\nmap(F)\|\leq b(F)\leq C_\pi\|\nmap(F)\|$ for $F\in\FSet$.

\section{Relations between forms of additivity}

In this section, we aim at a complete characterization of the relations among weak notions of additivity introduced in \S \ref{sec:space-maps}. A common theme in all these notions is the existence of an error that vanishes at infinity and measures how far apart from being additive a given set map is. Different ways of capturing this idea, led to different weak notions of additivity.

\subsection{Asymptotical additivity implies existence of additive realization}
\label{subsection:coboundaries}

Consider the subspace of {\bf coboundaries}
$$
\Cob := \left<\{v-\pi(g)v: v \in V, g \in G\}\right> \subseteq V,
$$
and that of {\bf weak coboundaries} $\wCob$, where the closure is taken with respect to the semi-norm $\|\cdot\|$. The space of weak coboundaries was characterized in \cite{bricenoiommi1}.

\begin{proposition}[{\cite[Corollary~3.5]{bricenoiommi1}}]
\label{cor:cob}
Given $v \in V$,  we have that $v \in \wCob$ if and only if $\limsup_{F \to G} \left\|A_F v\right\| = 0$.
\end{proposition}

Consider the quotient space $V/\overline{\Cob} = \{v + \overline{\Cob}: v \in V\}$. The quotient topology in $V/\overline{\Cob}$ is given by the semi-norm
$$
\|v + \overline{\Cob}\|_{\overline{\Cob}} = \inf_{u \in \overline{\Cob}} \|v+u\|.
$$
It turns out that $(V/\overline{\Cob}, \|\cdot\|_{\overline{\Cob}})$ is a Banach space \cite[\S 4.1]{bricenoiommi1}. The following is a particular case of \cite[Corollary~3.6]{bricenoiommi1}.

\begin{proposition}
\label{cor:LW}
For every $v \in V$,
$$
C_\pi^{-1}\limsup_{F \to G}\left\|A_F v\right\| \leq  \|v+\overline{\Cob}\|_{\overline{\Cob}} \leq \inf_{F \in \mathcal{F}(G)}\left\|A_F v\right\| \leq \limsup_{F \to G}\left\|A_F v\right\|.
$$
If, in addition, $\pi$ is isometric, then for every $v \in V$,
$$
\|v+\overline{\Cob}\|_{\overline{\Cob}} = \limsup_{F \to G} \left\|A_F v\right\| =  \inf_{F \in \mathcal{F}(G)} \left\|A_F v\right\|.
$$
\end{proposition}

It turns out that every asymptotically additive set map can be realized as an additive set map. This was proven in \cite[Theorem~1]{bricenoiommi1}.

\begin{theorem}[{\cite[Theorem~1]{bricenoiommi1}}]
\label{thm:cuneo}
If $\varphi \in \Maps$ is asymptotically additive, then there exists an additive set map $\psi \in \Maps$ such that
$$
\vertG{\varphi - \psi} = 0.
$$
\end{theorem}

\begin{corollary}[{\cite[Corollary~1]{bricenoiommi1}}]
\label{cor:cuneo}
Let $\varphi \in \Maps$. If
$$
\inf_{v \in V} \limsup_{F \to G} \left\|\frac{\varphi(F)}{|F|}-A_F v\right\| = 0,
$$
then there exists $v \in V$ such that 
\begin{equation}
\label{eq1}
\lim_{F \to G} \left\|\frac{\varphi(F)}{|F|}-A_F v\right\| = 0.
\end{equation}
\end{corollary}

Any element $v \in V$ that satisfies Equation \eqref{eq1} is called an {\bf additive realization} of $\varphi$.

\begin{remark}
\label{rem:additive}
Clearly, the converse of Theorem~\ref{thm:cuneo} holds, that is, if a set map $\varphi$ admits an additive realization, then $\varphi$ is asymptotically additive.    
\end{remark}

We will denote by $\witness$ the set of all additive realizations of $\varphi$. The following result was obtained in \cite{bricenoiommi1}.

\begin{lemma}[{\cite[Lemma~3.8]{bricenoiommi1}}]
\label{lem:witness}
If $v \in V$ is an additive realization of $\varphi$, then
$$
\witness = v + \wCob = \{v+u: u \in \wCob\}.
$$
\end{lemma}

\subsection{Asymptotical additivity implies almost additivity}

The following technical result shows that elements in $\NMaps_\R$---which measure how far from additivity a given almost additive set map is---are dominated by elements in $\NMaps_\R^+$ that are sub-additive over monotone invariant partitions.

\begin{lemma}
\label{lem:kernelsubad}
For every $b \in \NMaps_\R$, there exists a set map $b' \in \NMaps^+_\R$ such that
$$
|b(F)| \leq b'(F) \qquad \text{and} \qquad b'(F) \leq \sum_{E \in \Part} b'(E)
$$
for every set $F \in \FSet$ and every monotone invariant partition $\Part$ of $F$.
\end{lemma}

\begin{proof}
Let $b\in\NMaps_\R$. Pick a cofinal sequence $((K_n,\delta_n))_n$ in $\FSet\times\R_{>0}$ such that $K_0=\{1_G\}$, $K_n\subseteq K_{n+1}$, and $\delta_{n+1}\leq\delta_n$. Define
$$
n(F):=\sup\{n\in \N: F\text{ is $(K_n,\delta_n)$-invariant}\}
\in\N\cup\{\infty\},
$$
and
$$
r(n) := \sup\left\{\frac{|b(F)|}{|F|} : F \text{ is $(K_n,\delta_n)$-invariant} \right\} \qquad \text{for } n \in \N.
$$
Set $r(\infty)=0$. Notice that $0 \leq r(n+1) \leq r(n) \leq \vertsup{b}$ and, since $\vertG{b} = 0$, $r(n) \to 0$ as $n \to \infty$. If $n(F)=\infty$, then
$|b(F)|/|F|\leq r(n)$ for every $n\in \N$, and hence $b(F)=0$. Let $b'\colon \FSet \to \R$ be given by
$$
b'(F) := |F|r(n(F)).
$$
Since $|g \cdot F| = |F|$ and $n(g \cdot F) = n(F)$, $b'$ is $G$-equivariant and, since $n(F) \to \infty$ as $F \to G$ and $r(n) \to 0$ as $n \to \infty$,
$$
\vertG{b'}  =   \limsup_{F \to G} \frac{|b'(F)|}{|F|} = \limsup_{F \to G} r(n(F)) = 0,
$$
so $b' \in \NMaps^+_\R$. In addition, for every $F \in \FSet$, we have that $|b(F)| \leq |F|r(n(F)) = b'(F)$. Finally, given $F \in \FSet$ and a monotone invariant partition $\Part$ of $F$,
$$
b'(F)  =   \sum_{E \in \Part}|E|r(n(F)) \leq \sum_{E \in \Part}|E|r(n(E)) = \sum_{E \in \Part} b'(E),
$$
since $n(E) \leq n(F)$ for every $E \in \Part$.
\end{proof}

\begin{proposition}
\label{lem:asymp-almost}
If $\varphi \in \Maps$ is asymptotically additive, then $\varphi$ is almost additive.
\end{proposition}

\begin{proof}
Suppose that $\varphi \in \Maps$ is asymptotically additive. Then, by Theorem~\ref{thm:cuneo}, there exists an additive realization $\psi \in \AdMaps$ such that $\vertG{\psi-\varphi} = 0$. Let's denote $\psi - \varphi$ by $\nmap$. Since $\psi = \varphi + \nmap$ is additive,
$$
(\varphi+\nmap)(F) = \sum_{E \in \Part} (\varphi+\nmap)(E),
$$
for every set $F \in \FSet$ and every monotone invariant partition $\Part$ of $F$. Therefore,
$$
\varphi(F) - \sum_{E \in \Part} \varphi(E) = \sum_{E \in \Part} \nmap(E) - \nmap(F),
$$
so
$$
\left\|\varphi(F) - \sum_{E \in \Part} \varphi(E)\right\| = \left\|\sum_{E \in \Part} \nmap(E) - \nmap(F)\right\| \leq \sum_{E \in \Part} \|\nmap(E)\| + \|\nmap(F)\|.
$$

Let $b\colon \FSet \to \R$ be given by $b(F) = \sup_{g \in G} \|\nmap(g \cdot F)\|$. Note that $|b(F)| \leq |F|\vertsup{\nmap}$, so $b$ is well-defined, and $b(F) = b(g \cdot F)$ for every $g \in G$. Now, since $\nmap \in \NMaps$, we have that
$$
\vertG{b} = \limsup_{F \to G} \frac{|b(F)|}{|F|} =  \limsup_{F \to G} \frac{|\|\nmap(F)\||}{|F|} = 0,
$$
so $b\in \NMaps_\R$. Due to Lemma~\ref{lem:kernelsubad}, there exists $b' \in \NMaps_\R^+$ that is sub-additive over monotone invariant partitions and such that $\|\nmap(F)\| \leq b(F) \leq b'(F)$, and we conclude that
$$
\left\|\varphi(F) - \sum_{E \in \Part} \varphi(E)\right\| \leq \sum_{E \in \Part} \|\nmap(E)\| + \|\nmap(F)\| \leq \sum_{E \in \Part} (2b')(E).
$$
Since $2b' \in \NMaps_\R$, we have the result.
\end{proof}

\subsection{Almost additivity and residual finiteness imply asymptotical additivity}

A countable group $G$ is \textbf{residually finite} if there exists a decreasing sequence $(H_n)_n$ of finite index normal subgroups $H_n \unlhd G$ such that $\bigcap_{n \in \N} H_n=\{1_G\}$. In the amenable case, it is well-known (see \cite{cortez2014invariant,weiss_monotileable}) that we can always choose such a sequence $(H_n)_n$ with the additional property that its fundamental domains correspond to a F{\o}lner sequence $(F_n)$, thus providing a tiling $\{F_n h: h \in H_n\}$ of $G$. We will exploit this fact in the main theorem of this subsection. We will also make use of the following standard result, which holds in the general amenable setting.

\begin{lemma}
\label{lem:amenable}
For every $(K_0,\delta_0) \in \mathcal{F}(G) \times \R_{> 0}$, there exists $(K,\delta) \in \mathcal{F}(G) \times \R_{> 0}$ such that
$$
|\partial_{K_0}(F)| \leq \delta_0 |F|
$$
for every $(K,\delta)$-invariant set $F \in \mathcal{F}(G)$.
\end{lemma}

\begin{proof}
Consider $K = K_0 \cup K_0^{-1} \cup \{1_G\}$ and $\delta = \frac{\delta_0}{1+|K|}$. Take an arbitrary $(K,\delta)$-invariant set $F \in \FSet$. Since $1_G \in K$,
$$
\operatorname{Int}_{K}(F) \subseteq F \subseteq KF \qquad \text{and} \qquad |KF \setminus F| = |KF \Delta F| \leq \delta|F|.
$$
As $K = K^{-1}$, we have $$
KF = K^{-1}F = \operatorname{Int}_{K}(F) \sqcup \partial_{K}(F).
$$
Therefore,
$$
|\partial_{K_0}(F)| \leq |\partial_{K}(F)|  = |KF \setminus \operatorname{Int}_{K}(F)| = |KF \setminus F| + |F \setminus \operatorname{Int}_{K}(F)|.
$$
Finally, observing that $|F \setminus \operatorname{Int}_{K}(F)| \leq |K| |KF \setminus F|$, we obtain
$$
|\partial_{K_0}(F)| \leq (1+|K|)|KF \setminus F| \leq  (1+|K|)\delta|F| = \delta_0 |F|,
$$
completing the proof.
\end{proof}

\begin{lemma}
\label{lem:invariance-domination}
Let $E\in\FSet$ satisfy $hE=E$ only when $h=1_G$. Put $K=EE^{-1}$ and $\delta=1/(|E|(|E|+1))$. Every $(K,\delta)$-invariant set $F\in\FSet$ satisfies
$$
\frac{|K'F\Delta F|}{|F|}
\leq \frac{|K'E\Delta E|}{|E|}
\qquad \text{for every }K'\in\FSet.
$$
\end{lemma}

\begin{proof}
Fix a $(K,\delta)$-invariant set $F$ and an arbitrary $K'\in\FSet$. Consider the graph on $K'\cup\{1_G\}$ joining distinct vertices $a,b$ whenever $aE\cap bE\neq\emptyset$. Identify each component with its vertex set $C$. The sets $CE$ corresponding to distinct components are disjoint, since an intersection would imply that two vertices in different components
are adjacent.

For adjacent vertices $a,b$, we have $a^{-1}b\in K$, so
$$
|bF\setminus aF|
=|(a^{-1}b)F\setminus F|
\leq\delta|F|.
$$
For each component $C$, choose a rooted spanning tree and add the translates $cF$, $c\in C$, in an order in which parents precede their children. The root translate contributes $|F|$ elements. Each remaining translate adds at most $\delta|F|$ new elements, since its parent translate is already included in the union. Therefore, $|CF|\leq |F|+(|C|-1)\delta|F|$, which gives
$$
|CF|/|F|\leq1+(|C|-1)\delta.
$$

Fixing a point of $E$ shows that $|CE|\geq|C|$. If $|C|\geq2$, the triviality of the left stabilizer also gives $|CE|\geq |E|+1$, and hence
$$
|C|\delta
\leq\frac{|CE|}{|E|(|E|+1)}
\leq\frac{|CE|-|E|}{|E|}.
$$
Therefore, $|CF|/|F|\leq|CE|/|E|$ for every component, including singleton components.

Let $C_0$ contain $1_G$ and put $J=C_0\cap K'$. If $|C_0|\geq2$, a tree rooted at $1_G$ gives $|JF\setminus F|\leq(|C_0|-1)\delta|F|$. Moreover, $|F\setminus JF|\leq\delta|F|$, using a neighbor of $1_G$ when $1_G\notin K'$; otherwise this difference is empty. Since $C_0E=E\cup JE$, we obtain
$$
\frac{|JF\Delta F|}{|F|}
\leq |C_0|\delta
\leq\frac{|C_0E|-|E|}{|E|}
\leq\frac{|JE\Delta E|}{|E|}.
$$
If $C_0=\{1_G\}$, the two normalized symmetric differences are equal. Finally, disjointness of the sets $CE$ gives
$$
\frac{|K'F\Delta F|}{|F|} \leq\frac{|JF\Delta F|}{|F|} +\sum_{C\neq C_0}\frac{|CF|}{|F|}\leq\frac{|JE\Delta E|+\sum_{C\neq C_0}|CE|}{|E|} =\frac{|K'E\Delta E|}{|E|}.
$$
\end{proof}

\begin{theorem} \label{thm_zzc}
Suppose that $G$ is residually finite. Then, every almost additive set map $\varphi \in \Maps$ is asymptotically additive.
\end{theorem}

\begin{proof}
If $G$ is finite, set $v=\varphi(G)/|G|$. By equivariance, $\pi(g)v=v$ for every $g\in G$, and hence $S_Gv=\varphi(G)$. Since $F\to G$ implies $F=G$ eventually, $v$ is an additive realization of $\varphi$. We may therefore assume that $G$ is infinite. Since $G$ is amenable and residually finite, there are a F{\o}lner sequence $(F_n)_n$ in $\FSet$ and a decreasing sequence $(H_n)_n$ of finite index normal subgroups $H_n \unlhd G$ such that
$$
G=\bigsqcup_{t\in F_n}H_nt
 =\bigsqcup_{t\in F_n}tH_n
 =\bigsqcup_{h\in H_n}F_nh
\qquad \text{for all }n.
$$

We may assume that each $F_n$ contains $1_G$. We now modify these domains so that $hF_n=F_n$ only when $h=1_G$. Since $|F_n|\to\infty$, we may discard finitely many terms so that $|F_n|\geq3$ for all $n$. For each $n$, choose $t\in F_n\setminus\{1_G\}$ and $t'\in H_nt\setminus F_nF_n^{-1}F_n$, which is possible because $H_n$ is infinite, and set $F'_n=(F_n\setminus\{t\})\cup\{t'\}$. Since $H_nt'=H_nt$, the set $F'_n$ is again a fundamental domain containing $1_G$. If $h\neq1_G$ and $hF'_n=F'_n$, then $h$ maps at least one of the $|F_n|-1\geq2$ elements of $F_n\cap F'_n$ into $F_n$. Thus $h\in F_nF_n^{-1}$. Moreover, $ht'\in F_n$, so $t'\in F_nF_n^{-1}F_n$, a contradiction. The sequence $(F'_n)_n$ remains F{\o}lner since $|F'_n\Delta F_n|=2$ and $|F_n|\to\infty$. We now relabel $F'_n$ as $F_n$.

Fix $\epsilon_0 > 0$ and let $b \in \NMaps_\R$ such that $\varphi$ is $b$-almost additive. Let
$$
\epsilon = \frac{\epsilon_0}{1+\vertsup{b}+2\vertsup{\varphi}} > 0,
$$
and choose $n \in \N$ such that
$$
b(F_n) \leq |F_n|\epsilon.
$$

Set
$$
K=F_nF_n^{-1},
\qquad
\delta=\min\left\{
\frac{\epsilon}{1+|K|},
\frac{1}{|F_n|(|F_n|+1)}
\right\}.
$$
Since $K=K^{-1}$ and $1_G\in F_n\subseteq K$, the estimate in the proof of Lemma~\ref{lem:amenable} gives, for every $(K,\delta)$-invariant set $F\in\FSet$,
$$
\operatorname{Int}_{F_n}(F)\subseteq F
\qquad \text{and} \qquad
|\partial_{F_n}(F)|\leq\epsilon|F|.
$$

Fix a $(K,\delta)$-invariant set $F\in\FSet$.
By the choice of $\delta$ and
Lemma~\ref{lem:invariance-domination},
$$
\frac{|K'F\Delta F|}{|F|}
\leq\frac{|K'F_n\Delta F_n|}{|F_n|}
=\frac{|K'(F_ng)\Delta(F_ng)|}{|F_ng|}
$$
for every $K'\in\FSet$ and $g\in G$, where the equality follows because right multiplication by $g$ preserves cardinalities. For every $h\in G$ and $K'\in\FSet$, we also have
$$
\frac{|K'F\Delta F|}{|F|}
\leq |K'\Delta\{1_G\}|
=|K'\{h\}\Delta\{h\}|.
$$

Given $t\in F_n$, define
$$
\mathcal P_t
=\{F_ng:g\in H_nt\cap\operatorname{Int}_{F_n}(F)\}
\cup\{\{h\}:g\in H_nt\cap\partial_{F_n}(F),
                  \ h\in F_ng\cap F\}.
$$
This is a partition of $F$, since
$$
G=Gt=\bigsqcup_{c\in H_n}F_nct
    =\bigsqcup_{g\in H_nt}F_ng
$$
and
$$
\{g\in G:F_ng\cap F\neq\emptyset\}
=\operatorname{Int}_{F_n}(F)\sqcup\partial_{F_n}(F).
$$
The two comparisons above show that $\mathcal P_t$
is monotone invariant for every $t\in F_n$.
By almost additivity,
$$
\left\|
|F_n|\varphi(F)
-\sum_{t\in F_n}\sum_{E\in\mathcal P_t}\varphi(E)
\right\|
\leq
\sum_{t\in F_n}
\left\|\varphi(F)-\sum_{E\in\mathcal P_t}\varphi(E)\right\|
\leq
\sum_{t\in F_n}\sum_{E\in\mathcal P_t}b(E).
$$
Also, for every $t \in F_n$,
\begin{align*}
\sum_{E \in \mathcal{P}_t } b(E)    &   =   \sum_{g \in H_nt \cap \operatorname{Int}_{F_n}(F)} b(F_ng) +  \sum_{g\in H_nt\cap\partial_{F_n}(F)}
\sum_{h\in F_ng\cap F} b(\{h\})   \\
                                    &   \leq   \sum_{g \in H_nt \cap \operatorname{Int}_{F_n}(F)} b(F_n) + \sum_{g \in H_nt \cap \partial_{F_n}(F)} |F_n|\vertsup{b}  \\
                                    &   =   \left|H_nt \cap \operatorname{Int}_{F_n}(F)\right| b(F_n) + \left|H_nt \cap \partial_{F_n}(F)\right| |F_n|\vertsup{b}.
\end{align*}
Therefore, since $\{H_nt: t \in F_n\}$ is a partition of $G$,
\begin{align*}
\left\||F_n|\varphi(F) - \sum_{t \in F_n} \sum_{E \in \mathcal{P}_t} \varphi(E)\right\|                       & \leq  \sum_{t \in F_n} \left(\left|H_nt \cap \operatorname{Int}_{F_n}(F)\right| b(F_n) + \left|H_nt \cap \partial_{F_n}(F)\right| |F_n|\vertsup{b}\right)   \\
                            & = \left|\bigsqcup_{t \in F_n} H_nt \cap \operatorname{Int}_{F_n}(F)\right| b(F_n) + \left|\bigsqcup_{t \in F_n} H_nt \cap \partial_{F_n}(F)\right| |F_n|\vertsup{b}   \\
                            & = \left|\operatorname{Int}_{F_n}(F)\right| b(F_n) + \left|\partial_{F_n}(F)\right| |F_n|\vertsup{b}  \\
                            &   \leq  |F||F_n|\epsilon + \vertsup{b}|F||F_n|\epsilon = (1+\vertsup{b})|F||F_n|\epsilon,
\end{align*}
and, writing $\mathbbm{1}_F$ for the indicator of $F$,
\begin{align*}
\left\|\sum_{t \in F_n} \sum_{E \in \mathcal{P}_t } \varphi(E) - \sum_{g \in F} \varphi(F_ng) \right\|     & =  \left\|\sum_{t \in F_n} \sum_{E \in \mathcal{P}_t } \varphi(E) - \sum_{t \in F_n} \sum_{g \in  H_nt \cap F} \varphi(F_ng) \right\|  \\
                                            & \leq   \sum_{t\in F_n} \sum_{g\in H_nt\cap\partial_{F_n}(F)} \left\| \sum_{h\in F_ng\cap F}\varphi(\{h\})
-\mathbbm{1}_F(g)\varphi(F_ng)\right\|\\
                    &   \leq   \sum_{t \in F_n}\sum_{g \in H_nt \cap \partial_{F_n}(F)} 2\vertsup{\varphi}|F_n|\\
                    &   =  2\vertsup{\varphi}|F_n||\partial_{F_n}(F)| \leq  2\vertsup{\varphi}|F||F_n|\epsilon.
\end{align*}
Combining both inequalities, we obtain that
$$
\left\||F_n|\varphi(F) - \sum_{g \in F} \varphi(F_ng) \right\|  \leq (1+\vertsup{b}+2\vertsup{\varphi})|F||F_n|\epsilon = |F||F_n|\epsilon_0.
$$
Next, dividing by $|F||F_n|$, we have shown that there exists $v_{\epsilon_0} = \frac{\varphi(F_n)}{|F_n|} \in V$ such that
$$
\left\|\frac{\varphi(F)}{|F|} - A_F v_{\epsilon_0} \right\|  \leq \epsilon_0,
$$
where $A_F v_{\epsilon_0} = \frac{1}{|F|}\sum_{g \in F} \frac{\varphi(F_ng)}{|F_n|}$. Since this holds for every $(K,\delta)$-invariant set $F$, we have
$$
\limsup_{F \to G}
\left\|\frac{\varphi(F)}{|F|} - A_F v_{\epsilon_0} \right\|  \leq \epsilon_0.
$$
As $\epsilon_0$ is arbitrary, we conclude.
\end{proof}

\subsection{Riesz-almost additivity implies almost additivity}

\begin{proposition}
\label{lem:riesz-almost}
If $(V,\|\cdot\|,\ll)$ is a complete semi-normed Riesz space and $\varphi$ is Riesz-almost additive, then $\varphi$ is almost additive.
\end{proposition}

\begin{proof}
Let $\nmap \in \NMaps$ such that $\varphi$ is $\nmap$-Riesz-almost additive and let $b\colon \FSet \to \R$ be given by $b(F) = \sup_{g \in G} \|\nmap(g \cdot F)\|$. Note that $b\in \NMaps_\R$. Since $\|\cdot\|$ is a Riesz semi-norm, we have that for every $F \in \FSet$ and every monotone invariant partition $\Part$ of $F$,
$$
\left|\varphi(F) - \sum_{E \in \Part} \varphi(E)\right| \ll \sum_{E \in \Part} \nmap(E) \implies \left\|\varphi(F) - \sum_{E \in \Part} \varphi(E)\right\| \leq \sum_{E \in \Part} \left\|\nmap(E)\right\| \leq \sum_{E \in \Part} b(E),
$$
and we conclude.
\end{proof}

\begin{example}[Almost additive set map that is not Riesz-almost additive]
\label{exmp:counterexample}
Assume that $G$ is infinite. Let $([0,1], \mathcal{B}, \lambda)$ be the unit interval endowed with the Lebesgue measure. Suppose that $G \acts ([0,1], \lambda)$ is trivial and consider the associated Koopman representation $\koopman\colon G \to \Isom(L^1([0,1], \R))$ given by $\koopman(g) f = f$ for all $g \in G$ and $f \in L^1([0,1], \R)$. Take $\varphi\colon \FSet \to L^1([0,1], \R)$ given by
$$
\varphi(F) = |F|^2 \mathbbm{1}_{\left[0,\frac{1}{|F|\log(1+|F|)}\right]}    \qquad   \text{for } F \in \FSet.
$$
As $\|\varphi(F)\|_{L^1([0,1], \R)} \leq \frac{|F|}{\log(1+|F|)}$, it follows that $\varphi(F) \in L^1([0,1], \R)$ and $\vertsup{\varphi} < \infty$. Additionally, $\varphi$ is $G$-equivariant since it only depends on $|F|$. Let's prove that there exists $b \in \NMaps_\R^+$ such that $\varphi$ is $b$-almost additive. Indeed, given $F \in \FSet$ and a partition $\Part = \{E\}$ of $F$, observe that
$$
|F| = \sum_{E} |E| \qquad \text{and} \qquad |F|^2 \geq \sum_{E} |E|^2.
$$
This implies that
\begin{align*}
\left|\varphi(F) - \sum_E \varphi(E)\right| &   = (|F|^2-\sum_{E} |E|^2) \mathbbm{1}_{\left[0,\frac{1}{|F|\log(1+|F|)}\right]} + \sum_{E} |E|^2 \mathbbm{1}_{\left[\frac{1}{|F|\log(1+|F|)},\frac{1}{|E|\log(1+|E|)}\right]}    \\
&   \leq |F|^2 \mathbbm{1}_{\left[0,\frac{1}{|F|\log(1+|F|)}\right]} + \sum_{E} |E|^2 \mathbbm{1}_{\left[0,\frac{1}{|E|\log(1+|E|)}\right]},
\end{align*}
so
\begin{align*}
\left\|\varphi(F) - \sum_E \varphi(E)\right\|   &    \leq   \frac{|F|}{\log(1+|F|)} + \sum_E \frac{|E|}{\log(1+|E|)} \\
&   = \sum_E \left(\frac{|E|}{\log(1+|F|)}+\frac{|E|}{\log(1+|E|)}\right) \leq \sum_E \frac{2|E|}{\log(1+|E|)}.
\end{align*}
Thus, if $b\colon \FSet \to \R$ is given by $b(F) = \frac{2|F|}{\log(1+|F|)}$, we obtain that $b \in \NMaps_\R^+$ and $\varphi$ is $b$-almost additive. Next, suppose that $\varphi$ is $\xi$-Riesz-almost additive for some $\xi \in \NMaps$. Then, for every monotone invariant partition $\Part$ of $F$,
we would have
$$
\left|\varphi(F) - \sum_E \varphi(E)\right| \ll \sum_E \xi(E),
$$
so
$$
(|F|^2-\sum_{E} |E|^2) \mathbbm{1}_{\left[0,\frac{1}{|F|\log(1+|F|)}\right]} \ll \sum_E \xi(E)\mathbbm{1}_{\left[0,\frac{1}{|F|\log(1+|F|)}\right]},
$$
i.e.,
$$
|F|^2 \mathbbm{1}_{\left[0,\frac{1}{|F|\log(1+|F|)}\right]}(\omega) \leq \sum_E (|E|^2 + \xi(E)(\omega)) \mathbbm{1}_{\left[0,\frac{1}{|F|\log(1+|F|)}\right]}(\omega) \qquad   \lambda(\omega)\text{-a.e.}
$$
However, if we take the partition $\Part = \{\{g\}: g \in F\}$, we would obtain that
$$
|F|^2 \mathbbm{1}_{\left[0,\frac{1}{|F|\log(1+|F|)}\right]}(\omega) \leq |F| (1 + \xi(\{1_G\})(\omega)) \mathbbm{1}_{\left[0,\frac{1}{|F|\log(1+|F|)}\right]}(\omega) \qquad   \lambda(\omega)\text{-a.e.},
$$
so, if $|F| = n$ for some $n \geq 1$,
$$
(n-1) \mathbbm{1}_{\left[0,\frac{1}{n\log(n+1)}\right]}(\omega) \leq \xi(\{1_G\})(\omega) \mathbbm{1}_{\left[0,\frac{1}{n\log(n+1)}\right]}(\omega) \qquad   \lambda(\omega)\text{-a.e.}
$$
Thus, we obtain that
\begin{align*}
\|\xi(\{1_G\})\|_{L^1([0,1],\R)}     &   \geq    \int_{\left[0,1\right]}{\xi(\{1_G\})}d\lambda  \\
                    &   \geq  \int{\sum_{n \geq 1} (n-1)\mathbbm{1}_{\left[\frac{1}{(n+1)\log(n+2)},\frac{1}{n\log(n+1)}\right]}(\omega)}d\lambda(\omega)    \\
                                    &   = \sum_{n \geq 2} (n-1) \left(\frac{1}{n\log(n+1)} - \frac{1}{(n+1)\log(n+2)}\right)   \\
                                    &   \geq \sum_{n \geq 2} \frac{n}{2} \left(\frac{1}{n\log(n+1)} - \frac{1}{(n+1)\log(n+1)}\right)  \\
                                    &   = \frac{1}{2} \sum_{n \geq 3} \frac{1}{n\log n} = \infty,
\end{align*}
which contradicts the fact that $\xi(\{1_G\})$ belongs to $L^1([0,1], \R)$.
\end{example}

Example \ref{exmp:counterexample} exhibits a complete normed Riesz space that admits almost additive set maps that are not Riesz-almost additive. It seems very plausible to extend this fact to other semi-normed Riesz spaces. However, it is worthwhile to point out that almost additivity coincides with Riesz-almost additivity for the Banach lattice $(\R,|\cdot|,\leq)$ with the trivial representation, where the Riesz semi-norm condition is an equivalence. Moreover, under extra assumptions, it is possible to obtain a converse of Proposition~\ref{lem:riesz-almost} for more general Banach lattices.

Consider the Banach lattice $(\mathcal{B}(X),\|\cdot\|_\infty,\ll)$ of bounded real-valued functions over a set $X$ endowed with the supremum norm and the pointwise order. We say that a uniformly bounded representation
$\pi\colon G \to \Isom(\mathcal{B}(X))$ is \textbf{unital} if $\pi(g) \mathbbm{1} = \mathbbm{1}$ for every $g \in G$, where $\mathbbm{1}$ denotes the constant function equal to $1$. In this context, a set map $\varphi\colon \FSet \to \mathcal{B}(X)$ is \textbf{of constant type} if there exists real-valued set map $c\colon \FSet \to \R$ such that $\varphi(F) = c(F)\mathbbm{1}$ for every $F \in \FSet$. Notice that, since
$$
\|\varphi(F)\|_\infty = |c(F)|\|\mathbbm{1}\|_\infty = |c(F)|,
$$
we have that $\varphi$ is bounded if and only $c$ is bounded. Moreover, if $\pi$ is unital, 
$$
\pi(g) \varphi(F) = \pi(g) c(F)\mathbbm{1} = c(F) \pi(g) \mathbbm{1} = c(F) \mathbbm{1} \qquad \text{for every } g \in G,
$$
and, since $\varphi(g \cdot F) = c(g \cdot F)\mathbbm{1}$, we have that $\varphi$ is $G$-equivariant with respect to $\pi$ if and only if $c$ is $G$-equivariant with respect to the trivial representation $\pi_\R$. Thus, $\varphi$ belongs to $\Maps$ if and only if $c$ belongs to $\Maps_\R$, and $\varphi$ belongs to $\NMaps$ if and only if $c$ belongs to $\NMaps_\R$. In this setting, a set map will be called {\bf Riesz-almost additive of constant type} if it is $\nmap$-Riesz-almost additive for a set map $\nmap \in \NMaps$ of constant type.

\begin{proposition}
\label{prop:riesz-almost-constant}
Let $(\mathcal{B}(X),\|\cdot\|_\infty,\ll)$ be the Banach lattice of bounded real-valued functions over a set $X$. Consider a unital, uniformly bounded representation $\pi\colon G \to \Isom(\mathcal{B}(X))$ and a set map $\varphi \in \Maps$. Then, $\varphi$ is Riesz-almost additive of constant type if and only if $\varphi$ is almost additive.
\end{proposition}

\begin{proof}
If $\varphi$ is almost additive, there exists $b \in \NMaps_\R$ such that $\varphi$ is $b$-almost additive. Then, for every $x \in X$,
$$
\left|\varphi(F)-\sum_{E \in \mathcal{P}} \varphi(E)\right|(x) \leq \left\|\varphi(F)-\sum_{E \in \mathcal{P}} \varphi(E)\right\|_\infty \leq \sum_{E \in \mathcal{P}} b(E) = \sum_{E \in \mathcal{P}} b(E)\mathbbm{1}(x).
$$
Letting $\nmap\colon \FSet \to \mathcal{B}(X)$ be given by $\nmap(F) = b(F)\mathbbm{1}$ for every $F \in \FSet$, we get that $\nmap$ belongs to $\NMaps$, $\varphi$ is $\nmap$-Riesz-almost additive, and $\nmap$ is of constant type, so $\varphi$ is Riesz-almost additive of constant type. Here, $G$-equivariance of $\nmap$ follows from $G$-invariance of $b$ and the fact that $\pi$ is unital. The converse directly follows from Proposition~\ref{lem:riesz-almost}.
\end{proof}

\begin{remark}
\label{rem:riesz-riesz}
Notice that if $X$ is a topological space, the closed subspace of bounded and continuous real-valued functions $\mathcal{C}_b(X)$ induces a Banach lattice as well, and we can have an analogue of Proposition~\ref{prop:riesz-almost-constant} restricted to $\mathcal{C}_b(X)$. Clearly, in both contexts, every set map that is Riesz-almost additive of constant type, is Riesz-almost additive.
\end{remark}

\subsection{Summary of implications}

The following result completely characterizes the relations between the different weak forms of additivity discussed in this section.

\theoremone*

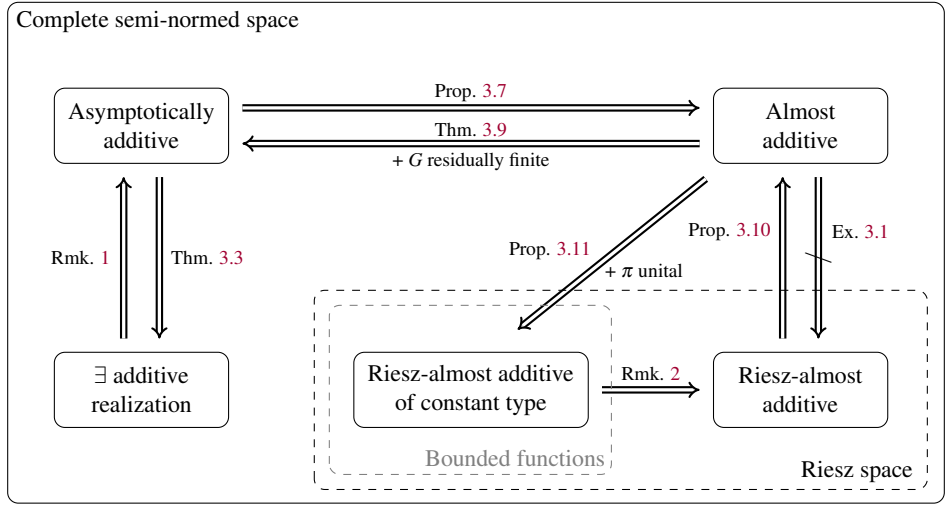
\begin{figure}[ht]
\centering
\begin{tikzpicture}[baseline]
    \node[draw, rounded corners, inner xsep=0.5cm, inner ysep=1.0cm](box){
\begin{tikzcd}[
    row sep=2.5cm, column sep=1.65cm,
    cells={nodes={draw, rounded corners, minimum width=1.5cm, minimum height=1cm}}
]
\parbox{2.0cm}{\centering \small Asymptotically additive}
    \arrow[rr, shift left=1.5ex, Rightarrow, thick, shorten >=0.5em, shorten <=0.5em,"\text{Prop. \ref{lem:asymp-almost}}"] 
    \arrow[rr, shift right=1.5ex, Leftarrow, thick, shorten >=0.5em, shorten <=0.5em, "\text{+ $G$ residually finite}"' {yshift=-1pt},"\text{Thm. \ref{thm_zzc}}"] 
    \arrow[d, shift left=1.5ex, Rightarrow, thick, shorten >=0.5em, shorten <=0.5em, swap,"\text{~Thm. \ref{thm:cuneo}}"']
    \arrow[d, shift right=1.5ex, Leftarrow, thick, shorten >=0.5em, shorten <=0.5em, swap, "\text{Rmk. \ref{rem:additive}}~"]
&  &
\parbox{2.0cm}{\centering \small Almost additive}
    \arrow[d, shift left=1.5ex, Rightarrow, thick, shorten >=0.5em, shorten <=0.5em, swap, negate, "~\text{Ex. \ref{exmp:counterexample}}"' {yshift=10pt}]
    \arrow[d,shift right=1.5ex, Leftarrow, thick, shorten >=0.5em, shorten <=0.5em, swap, "\text{Prop. \ref{lem:riesz-almost}~}" {yshift=10pt}]
    \arrow[dl, shift right=1.5ex, Rightarrow, thick, shorten >=0.5em, shorten <=1.75em,"\text{Prop. \ref{prop:riesz-almost-constant}}"' {xshift=-10pt, yshift=-8.5pt},"\text{+ $\pi$ unital}" {xshift=-10pt, yshift=-5pt}]\\
\parbox{2.0cm}{\centering \small $\exists$ additive realization} & 
\parbox{2.8cm}{\centering \small Riesz-almost additive of constant type}
\arrow[r, Rightarrow, thick, shorten >=0.5em, shorten <=0.5em,"\text{Rmk. \ref{rem:riesz-riesz}}" {yshift=1pt}]
&
\parbox{2.0cm}{\centering \small Riesz-almost additive}
\end{tikzcd}
    };
\node[anchor=north west, font=\large] at (box.north west) {\small Complete semi-normed space};

\draw[dashed,rounded corners] 
        (-1.65-0.5, -2.35-0.8) rectangle (5.5+0.5, -1.3+0.8);
\draw[dashed,rounded corners,gray] 
        (-1.65-0.3, -2.35-0.6) rectangle (1.5+0.3, -1.3+0.6);
\node[anchor=north west, font=\large] at (4.2, -2.65) {\small Riesz space};
\node[anchor=north west, font=\large] at (-0.8, -2.5) {\small \textcolor{gray}{Bounded functions}};
\end{tikzpicture}
\caption{Diagram of implications for Theorem~\ref{thm:main}.}
\label{diagram}
\end{figure}

\section{Ergodic theorems for non-additive set maps}

By means of Theorem~\ref{thm:cuneo} and Corollary~\ref{cor:cuneo}, to every asymptotically additive set map we can associate an additive realization. This fact allows us to prove both mean and pointwise versions of ergodic theorems for asymptotically additive set maps by reducing them to their additive versions.

\subsection{A mean ergodic theorem for non-additive set maps}
\label{sec:mean-ergodic}

Generalizations of the mean ergodic theorem have been obtained for a wide range of groups, actions, and spaces. Our following result builds up on these together with our study of weak forms of additivity. In this, we prove a non-additive version of results obtained for the additive setting.

\theoremtwo*

\begin{remark}
\label{rem:meanerg}
If in the statement of Theorem~\ref{thm:meanasymp} we replace the asymptotically additive set map $\varphi$ with the additive set map $F \mapsto S_Fv$ for some $v \in V$, this result was obtained in \cite[Theorem~4.4]{pogorzelski2013almost} for Banach spaces, and in \cite[Theorem~4.22]{kerr2016ergodic} for unitary representations on Hilbert spaces. The uniformly bounded case on Hilbert spaces follows from the Banach space result, since Hilbert spaces are reflexive and the compactness assumption is therefore automatic.
\end{remark}

\begin{proof}[Proof of Theorem~\ref{thm:meanasymp}]
By Corollary~\ref{cor:cuneo}, there exists $v \in V$ such that
$$
\lim_{F \to G} \left\|\frac{\varphi(F)}{|F|}-A_F v\right\| = 0.
$$
By Remark~\ref{rem:meanerg}, there is a continuous projection $P$ of $V$ onto $\Fix(\pi)$ such that
$$
\lim_{F \to G} \left\|A_F v - Pv\right\| = 0.
$$
Combining both equations yields the asserted convergence. If $V$ is a Hilbert space and $\pi$ is isometric, then $\langle A_Fv,w\rangle=\langle v,w\rangle$ for every $v\in V$ and $w\in\Fix(\pi)$. Passing to the limit gives $\langle v-Pv,w\rangle=0$ for every $v\in V$ and $w\in\Fix(\pi)$. Thus $P$ is the orthogonal projection onto $\Fix(\pi)$.
\end{proof}

\begin{remark}
Krein-Smulian theorem \cite[\S II.2, Theorem~11]{diesteluhl} establishes that the closed convex hull of a weakly compact subset of a Banach space is weakly compact, so in Theorem~\ref{thm:meanasymp} it suffices to assume that the orbit $\left\{\pi(g) v : g \in G\right\}$ is relatively weakly compact for each $v \in V$. Moreover, since $\pi$ is uniformly bounded, if $V$ is reflexive, the compactness assumption is automatically fulfilled due to the Banach-Alaoglu Theorem. It is also worthwhile mentioning that for $1 \leq p \leq \infty$, if $(\Omega, \mathcal{B}, \mu)$ is a finite measure space, the Bochner space $L^p(\Omega, Y)$ is reflexive if and only if both $Y$ and $L^p(\Omega, \R)$ are reflexive \cite[\S IV.1, Corollary~2]{diesteluhl}.
\end{remark}

\begin{remark}
Notice that in Theorem~\ref{thm:meanasymp}, we can replace the assumption that $(V,\|\cdot\|)$ is a Banach space with that of a complete semi-normed space.
Indeed, the space $(V,\|\cdot\|)$ is complete if and only if $(V/K,\|\cdot\|_K)$ is a Banach space, where $K$ is the kernel of the seminorm and $\|\cdot\|_K$ is the norm given by $\|v+K\|_K=\|v\|$ for any $v\in V$.
Assuming the compactness hypothesis of the theorem for the induced representation on $V/K$, we obtain a continuous projection $\overline{P}$ on $V/K$. In this case, we can conclude that there exists $u\in p^{-1}(\overline{P}(V/K))$ such that
$$
\lim_{F\to G}
\left\|\frac{\varphi(F)}{|F|}-u'\right\|=0
$$
for any $u'\in u+K$, where $p\colon V\to V/K$ denotes the quotient map.
\end{remark}

Interestingly, we can appeal to Theorem~\ref{thm:main} to prove a version of Theorem~\ref{thm:meanasymp} for almost additive set maps provided the group is residually finite. This allows us to recover and provide a new perspective on previous results.

In \cite{pogorzelski2013almost}, Pogorzelski called a set map $b\colon \mathcal{F}(G) \to \R_{\geq 0}$ to be a \emph{boundary term} if, for all $E, F \in \FSet$,
\begin{enumerate}
\item $b(g \cdot F)=b(F)$ for all $g \in G$;
\item $\lim_{F \to G} \frac{b(F)}{|F|}=0$;
\item $b(E \cap F) \leq b(E)+b(F)$, $b(E \cup F) \leq b(E)+b(F)$ and $b(E \setminus F) \leq b(E)+b(F)$.
\end{enumerate}
Here we use the convention $b(\emptyset)=0$. Based on this, he called a set map $\varphi \in \Maps$ to be \emph{almost additive} if, for some boundary term $b$, it holds that
\begin{equation}
\label{eqn:almost}
\left\|\varphi(F)-\sum_{E \in \mathcal{P}} \varphi(E)\right\| \leq \sum_{E \in \mathcal{P}} b(E)
\end{equation}
for all $F \in \FSet$ and every partition $\mathcal{P} \subseteq \FSet$ of $F$. We would like to clarify how this notion fits in our context. First, if $b$ is a boundary term, then
$$
\frac{|b(F)|}{|F|} = \frac{b(F)}{|F|} \leq \sum_{g \in F} \frac{b(\{g\})}{|F|} = |F| \frac{b(\{1_G\})}{|F|} = b(\{1_G\}),
$$
so $\vertsup{b} = b(\{1_G\}) < \infty$. Since also $b$ satisfies (1) and (2), we have that $b$ belongs to  $\NMaps_\R$. Moreover, our definition of almost additive requires Equation \eqref{eqn:almost} to hold only for monotone invariant partitions. Therefore, every almost additive set map in the sense of Pogorzelski is almost additive in our sense but not necessarily vice versa.

Using his definition of almost additivity, and requiring an extra technical condition called \emph{tiling admissibility} upon boundary terms, Pogorzelski studied weakly measurably group representations of general unimodular groups having strong F\o lner sequences. Since discrete amenable groups satisfy all these properties, we present here a version of \cite[Theorem~5.7]{pogorzelski2013almost} adapted to our context.

\begin{theorem}[{\cite[Theorem~5.7]{pogorzelski2013almost}}]
\label{thm:meanset}
Let $G$ be a countable discrete amenable group, $(V,\|\cdot\|)$ a Banach space, and $\pi\colon G \to \Isom(V)$ a uniformly bounded representation.  Suppose that, for each $v \in V$, the convex hull $\operatorname{co}\left\{\pi(g) v : g \in G\right\}$ is relatively weakly compact. Then, for every bounded and $G$-equivariant set map $\varphi\colon \FSet \to V$ that is almost additive for some tiling-admissible boundary term $b$, there exists $u \in \Fix(\pi)$ such that 
$$
\lim_{F \to G}\left\|\frac{\varphi(F)}{|F|} - u\right\|=0.
$$
\end{theorem}

Although the setting considered by Pogorzelski is much more general, when restricted to the countable, discrete, residually finite case, our theory allows us to obtain the same conclusions under weaker assumptions. As we already explained, our definition of almost additivity is weaker and, moreover, we do not require any extra assumption on the set map $b$ such as tiling admissibility. Then, we can see that Theorem~\ref{thm:meanasymp} implies Theorem~\ref{thm:meanset}, as almost additivity in the sense of Pogorzelski implies almost additivity, and almost additivity implies asymptotical additivity by Theorem~\ref{thm:main}. We have therefore established, assuming residual finiteness, that almost additive set maps in the sense of Pogorzelski have additive realizations. In this way we reduce the almost additive case to the more classical additive one, providing a more transparent proof.

\subsection{A pointwise ergodic theorem for non-additive set maps} \label{sec:point}

Let $G$ be a countable discrete amenable group acting on a probability space $(\Omega, \mathcal{B}, \mu)$ by measure-preserving transformations, and let $(F_n)_n$ be a tempered F{\o}lner sequence. Lindenstrauss' pointwise ergodic theorem \cite[Theorem~2.1]{lindenstrauss2001pointwise} establishes that for any $f \in L^1(\Omega,\R)$, there exists $\overline{f} \in L^1(\Omega,\R)$ that is $G$-invariant (for the Koopman representation $\koopman$) such that
$$
\lim_{n \to \infty} A_{F_n}f(\omega)=\overline{f}(\omega) \qquad   \mu(\omega)\text{-a.e.}
$$
This far-reaching generalization of Birkhoff's theorem has been subsequently extended to different contexts \cite{bowen_nevo,nevo,pogorzelski2013almost}.

Given $p \geq 1$, we let $L^p_\infty(\Omega, Y) := L^p(\Omega, Y) \cap L^\infty(\Omega, Y)$ and define the Banach space $(L^p_\infty(\Omega, Y),\|\cdot\|_{L^p_\infty(\Omega, Y)})$, where $\|\cdot\|_{L^p_\infty(\Omega, Y)} = \|\cdot\|_{L^p(\Omega, Y)} + \|\cdot\|_{L^\infty(\Omega, Y)}$. Notice that if $Y$ is a Banach lattice, so is $L^p_\infty(\Omega, Y)$ under the pointwise order. We have the following non-additive version of \cite[Theorem~6.8]{pogorzelski2013almost} adapted to $L^p_\infty(\Omega, Y)$ instead of $L^p(\Omega, Y)$, framed in the setting of group representations.

\theoremthree*

\begin{remark}
\label{rem:pogorzelski6-8}
If in the statement of Theorem~\ref{thm:pointwise-asymp} we replace the asymptotically additive set map $\varphi$ with the additive set map $F \mapsto S_Fv$ for some $v \in V$, the result was obtained in \cite[Theorem~6.8]{pogorzelski2013almost} for $L^p(\Omega, Y)$. As pointed out in \cite{pogorzelski2013almost}, if we choose $B = 1$ and $\theta = \operatorname{id}_G$, the condition
$$
\left\|(\pi(g)f)(\omega)\right\|_Y \leq B\left\|f\left((\theta(g))^{-1} \cdot \omega\right)\right\|_Y  \qquad   \mu(\omega)\text{-a.e.}
$$
is satisfied with an equality and everywhere by the Koopman representation given by
$$
(\koopman(g) f)(\omega)=f(g^{-1} \cdot\omega)
$$ for $f \in L^p(\Omega, Y)$, $g \in G$, and $\omega \in \Omega$.
\end{remark}

\begin{proof}[Proof of Theorem~\ref{thm:pointwise-asymp}]
By Corollary~\ref{cor:cuneo}, there exists $f \in L^p_\infty(\Omega, Y)$ such that 
$$
\lim_{n \to \infty} \left\|\frac{\varphi(F_n)}{|F_n|} - A_{F_n} f\right\|_{L^p_\infty(\Omega, Y)} \leq \lim_{F \to G} \left\|\frac{\varphi(F)}{|F|} - A_F f\right\|_{L^p_\infty(\Omega, Y)} = 0,
$$
and as convergence in $L^p_\infty(\Omega, Y)$ implies convergence in $L^\infty(\Omega, Y)$, which at the same time implies $\mu$-almost everywhere convergence, there exists $\tilde{\Omega}_1 \subseteq \Omega$ such that $\mu(\Omega \setminus \tilde{\Omega}_1) = 0$ and
$$
\lim_{n \to \infty} \left\|\frac{\varphi(F_n)}{|F_n|}(\omega) - A_{F_n} f(\omega)\right\|_Y = 0 \qquad   \text{for every } \omega \in \tilde{\Omega}_1.
$$

The pointwise domination assumption and measure preservation give
$$
\|\pi(g)h\|_{L^p(\Omega,Y)}
\leq B\|h\|_{L^p(\Omega,Y)}
\qquad \text{for all }g\in G
\text{ and }h\in L^p_\infty(\Omega,Y).
$$
Since $1\leq p<\infty$, the space $L^p_\infty(\Omega,Y)$ is dense in $L^p(\Omega,Y)$. Thus each $\pi(g)$ extends uniquely to a bounded operator on $L^p(\Omega,Y)$ with operator norm at most $B$. By density, these extensions form a uniformly bounded representation, still denoted by $\pi$. The same pointwise domination holds for the extensions, as follows by approximating with bounded truncations and passing to an almost everywhere convergent subsequence of their images under $\pi(g)$.

Since $L^p_\infty(\Omega, Y) \subseteq L^p(\Omega, Y)$, and appealing to Remark~\ref{rem:pogorzelski6-8}, there exist a $\pi$-invariant function
$\overline{f}\in L^p(\Omega,Y)$
and $\tilde{\Omega}_2\subseteq\Omega$ with $\mu(\Omega \setminus \tilde{\Omega}_2) = 0$ such that
$$
\lim_{n \to \infty}\left\|A_{F_n}f(\omega) - \overline{f}(\omega)\right\|_Y=0    \qquad   \text{for every } \omega \in \tilde{\Omega}_2.
$$
Therefore, since
$$
\left\|\frac{\varphi(F_n)}{|F_n|}(\omega) - \overline{f}(\omega)\right\|_Y  \leq \left\|\frac{\varphi(F_n)}{|F_n|}(\omega) - A_{F_n} f(\omega)\right\|_Y + \left\|A_{F_n}f(\omega) - \overline{f}(\omega)\right\|_Y,
$$
after taking the limit $n \to \infty$, we conclude that
$$
\lim_{n \to \infty}\left\|\frac{\varphi(F_n)}{|F_n|}(\omega) - \overline{f}(\omega)\right\|_Y = 0  \qquad \text{for every } \omega \in \tilde{\Omega}_1 \cap \tilde{\Omega}_2,
$$
i.e., the convergence holds almost everywhere. It remains to show that $\overline{f}$ belongs to $L^p_\infty(\Omega, Y)$. Since $G$ is countable and $f\in L^p_\infty(\Omega,Y)$,
there exists $\tilde{\Omega}_3\subseteq\Omega$ with $\mu(\Omega\setminus\tilde{\Omega}_3)=0$ such that the action identities hold there for all $g,h\in G$ and
$$
\|f(\omega)\|_Y \leq \|f\|_{L^\infty(\Omega, Y)} \qquad   \text{for every } \omega \in \tilde{\Omega}_3.
$$
Then, given $g \in G$,
$$
\|f(g \cdot \omega)\|_Y \leq \|f\|_{L^\infty(\Omega, Y)} \qquad   \text{for every } \omega \in g^{-1} \cdot \tilde{\Omega}_3.
$$
By countability of $G$, we may further choose
$\tilde{\Omega}_3$ so that the pointwise domination
inequality holds there for the chosen $f$
and every $g\in G$. In particular, we have that
$$
\|A_F f (\omega)\|_Y \leq \frac{1}{|F|} \sum_{g \in F} \|(\pi(g^{-1})f)(\omega)\|_Y \leq \frac{B}{|F|} \sum_{g \in F} \|f(\theta(g) \cdot \omega)\|_Y \leq B \|f\|_{L^\infty(\Omega, Y)}
$$
for every $F\in\FSet$ and every
$\omega\in\tilde{\Omega}_3
\cap\bigcap_{g\in F}
\theta(g)^{-1}\cdot\tilde{\Omega}_3$. This implies that
$$
\|\overline{f}(\omega)\|_Y = \|\lim_{n \to \infty} A_{F_n} f (\omega)\|_Y = \lim_{n \to \infty} \|A_{F_n} f (\omega)\|_Y \leq \lim_{n \to \infty} B \|f\|_{L^\infty(\Omega, Y)} = B \|f\|_{L^\infty(\Omega, Y)}
$$
for every $\omega \in \tilde{\Omega}_2 \cap \bigcap_{g \in G} g^{-1} \cdot \tilde{\Omega}_3$. Since $\tilde{\Omega}_2 \cap \bigcap_{g \in G} g^{-1} \cdot \tilde{\Omega}_3$ is a countable intersection of full measure sets (as $G$ acts by measure-preserving transformations), we conclude that $\overline{f}$ is essentially bounded, i.e., $\overline{f} \in L^p_\infty(\Omega, Y)$.
\end{proof}

\begin{remark}
Notice that in Theorem~\ref{thm:pointwise-asymp}, which deals with the asymptotically additive case, we cannot freely replace $L^p_\infty(\Omega, Y)$ by $L^p(\Omega, Y)$, in contrast to the merely additive case. Indeed, take $G=\Z$, $Y=\R$, and let $(\Omega,\mathcal B,\mu)$ be the unit interval endowed with Lebesgue measure. Let the action and the representation be
trivial, and consider the tempered F{\o}lner sequence
$F_n=[0,n]\cap\Z$, indexed by $n\in\N$.

It suffices to consider a sequence $(f_n)_n$ that converges to $f \equiv 0$ in $L^p(\Omega, Y)$ but does not converge almost everywhere along the subsequence $(f_{|F_n|})_n$. This is the case of the so-called
\emph{typewriter sequence}, defined by $f_0=0$ and
$$
f_n:=\mathbbm{1}_{\left[
\frac{n-2^k}{2^k},\frac{n-2^k+1}{2^k}
\right]}
$$
whenever $k\geq0$ and $2^k\leq n<2^{k+1}$. For every $1\leq p<\infty$, we have $\|f_n\|_{L^p(\Omega,Y)}=2^{-k/p}\to0$ as $n\to\infty$,
whereas $(f_n)_n$ does not converge almost everywhere. Then, if $\varphi\colon \FSet \to L^p(\Omega, Y)$ is the set map given by $\varphi(F) = |F|f_{|F|}$, we have that $\varphi$
is bounded (since $\frac{\varphi(F)}{|F|} = f_{|F|}$ and $(f_n)_n$ converges in $L^p(\Omega, Y)$) and $G$-equivariant (since $|Fg^{-1}| = |F|$). Moreover, since
$|F|\to\infty$ as $F\to G$,
$$
\lim_{F \to G} \left\|\frac{\varphi(F)}{|F|}-A_F f\right\|_{L^p(\Omega, Y)} = \lim_{F \to G} \left\|f_{|F|}\right\|_{L^p(\Omega, Y)} = 0,
$$
so $\varphi$ is asymptotically additive and $f$ is an additive realization of $\varphi$. However, there is no $\overline{f} \in L^p(\Omega, Y)$ such that
$$
\lim_{n \to \infty} \left\|\frac{\varphi(F_n)}{|F_n|}(\omega)-\overline{f}(\omega)\right\|_Y = 0   \qquad   \mu(\omega)\text{-a.e.},
$$
because
$$
\lim_{n \to \infty} \left\|f_{|F_n|}(\omega)-\overline{f}(\omega)\right\|_Y = \lim_{n \to \infty} \left\|\frac{\varphi(F_n)}{|F_n|}(\omega)-\overline{f}(\omega)\right\|_Y
$$
and $(f_n)_n$ does not converge almost everywhere.
\end{remark}

In the particular case of the Koopman representation, which is naturally adapted to the $L^1_\infty(\Omega, Y)$ case, we now prove that the integral of weak coboundaries is zero (see \S \ref{subsection:coboundaries}).

\begin{lemma}
\label{lem:integralcoboundary}
Suppose that $(\Omega, \mathcal{B}, \mu)$ is a measure space and $G \acts (\Omega, \mu)$. If $\koopman\colon G \to \Isom(L^1_\infty(\Omega, Y))$ is the Koopman representation, then
$$
\int_\Omega f d\mu = 0  \qquad   \text{for every } f \in \wCob.
$$
\end{lemma}

\begin{proof}
If $f$ is a coboundary, that is, if $f \in \Cob$, then there exist $f_1,\dots,f_\ell \in L^1_\infty(\Omega, Y)$, $\lambda_1,\dots,\lambda_\ell \in \R$, and $g_1,\dots,g_\ell \in G$ such that
$$
f = \sum_{i=1}^\ell \lambda_i(f_i - \koopman(g_i) f_i),
$$
so
$$
\int_\Omega f d\mu = \sum_{i=1}^\ell \lambda_i \int_\Omega(f_i -  \koopman(g_i) f_i)d\mu = \sum_{i=1}^\ell \lambda_i \left(\int_\Omega f_i d\mu - \int_\Omega f_i d\mu\right) = 0. 
$$
If $f \in \wCob$, for every $\epsilon > 0$ there exists $f_\epsilon \in \Cob$ such that $\|f - f_\epsilon\|_{L^1(\Omega, Y)} \leq \epsilon$. Then, appealing to \cite[\S II.2, Theorem~4]{diesteluhl},
$$
\left\|\int_\Omega f d\mu\right\|_Y  = \left\|\int_\Omega f- \int_\Omega f_\epsilon d\mu\right\|_Y \leq  \int_\Omega \|f - f_\epsilon\|_Y d\mu \leq \epsilon,
$$
and since $\epsilon$ was arbitrary, we conclude that
$$
\left\|\int_\Omega f d\mu\right\|_Y = 0, \qquad \text{so} \qquad \int_\Omega f d\mu = 0.
$$
\end{proof}

If $(\Omega, \mathcal{B}, \mu)$ is a probability space, we say that $G$ \textbf{acts ergodically} if for every $A \in \mathcal{B}$ such that $g \cdot A = A$ for every $g \in G$, we have that $\mu(A) \in \{0,1\}$.

\begin{corollary}
\label{cor:pointwise-asymp}
Let $G$ be a countable discrete amenable group that acts ergodically on a probability space $(\Omega, \mathcal{B}, \mu)$ by measure-preserving transformations, $(Y,\|\cdot\|_Y)$ a reflexive Banach space, and $\koopman\colon G \to \Isom(L^p_\infty(\Omega, Y))$ the Koopman representation for $1 \leq p < \infty$. Then, if $(F_n)_n$ is a tempered F{\o}lner sequence in $G$, for each bounded and $G$-equivariant asymptotically additive set map $\varphi\colon \FSet \to L^p_\infty(\Omega, Y)$,
$$
\lim_{n \to \infty} \frac{\varphi(F_n)}{|F_n|}(\omega) = \int{f}d\mu    \qquad   \mu(\omega)\text{-a.e.},
$$
where $f$ is any additive realization of $\varphi$.
\end{corollary}

\begin{proof}
From the proof of Theorem~\ref{thm:pointwise-asymp}, there exist $f,\overline{f} \in L^p_\infty(\Omega, Y)$ such that
$$
\lim_{n \to \infty}\frac{\varphi(F_n)}{|F_n|}(\omega) = \lim_{n \to \infty} A_{F_n}f(\omega) = \overline{f}(\omega)   \qquad   \mu(\omega)\text{-a.e.}
$$
and $\koopman(g) \overline{f}=\overline{f}$ for all $g \in G$. Since $G$ acts ergodically, $\overline{f}$ is constant $\mu$-almost everywhere, so
$$
\overline{f}(\omega) = \int_\Omega{\overline{f}}d\mu   \qquad   \mu(\omega)\text{-a.e.}
$$

Since, for every $n$,
$$
\|A_{F_n} f (\omega)\|_Y \leq \frac{1}{|F_n|} \sum_{g \in F_n} \|f(g \cdot \omega)\|_Y \leq \|f\|_{L^\infty(\Omega, Y)} \qquad   \mu(\omega)\text{-a.e.}
$$
we can appeal to the Dominated Convergence theorem for Bochner integrals \cite[\S II.2, Theorem~3]{diesteluhl}, and obtain that
$$
\int_\Omega{\overline{f}}d\mu = \int_\Omega{\lim_{n \to \infty} A_{F_n}f}d\mu = \lim_{n \to \infty} \int_\Omega{A_{F_n}f}d\mu = \lim_{n \to \infty} \int_\Omega{f}d\mu = \int_\Omega{f}d\mu,
$$
where in the third equality we use that $G$ acts by measure-preserving transformations.
\end{proof}

\begin{remark}
Observe that, if $(\Omega, \mathcal{B}, \mu)$ is a probability space, then $1 \leq p \leq q < \infty$ implies that $L^q_\infty(\Omega, Y) \subseteq L^p_\infty(\Omega, Y)$. In particular, in light of Lemma~\ref{lem:witness}, Lemma~\ref{lem:integralcoboundary} is consistent with Corollary~\ref{cor:pointwise-asymp} in the sense that the set of additive realizations $\witness$ is of the form $f + \wCob$. 
\end{remark}

\section{Almost additivity in the sequential case}
\label{sec:folner}

In many situations, it is common to work with families of set maps $\varphi$ that are partially defined. For instance, given a F{\o}lner sequence $(F_n)_n$, we may be interested in studying limits only along it, such as
$$
\limsup_{n \to \infty}\left\|\frac{\varphi(F_n)}{|F_n|}\right\|.
$$
Here we show how our formalism is a natural and more encompassing way of studying this kind of situations.

Let $\mathcal{S}$ be a $G$-invariant subset of $\FSet$ (i.e., $G \cdot \mathcal{S} \subseteq \mathcal{S}$) such that for every $(K,\delta) \in \mathcal{F}(G) \times \R_{> 0}$ there exists a $(K,\delta)$-invariant set $F \in \mathcal{S}$. If $(F_n)_n$ is a F{\o}lner sequence, the natural subset $\mathcal{S}$ to consider is the $G$-orbit of $(F_n)_n$, namely 
$$
\mathcal{S} = \{g \cdot F_n : g \in G,~n \in \N\}.
$$
A map $\varphi\colon\{F_n\}_n\to V$ admits a $G$-equivariant extension to $\mathcal S$ if and only if, for all $g,h\in G$ and $m,n\in\N$, if $g\cdot F_n=h\cdot F_m$, then $g\cdot\varphi(F_n)=h\cdot\varphi(F_m)$. In this case, the extension is uniquely defined by
$$
\varphi(g\cdot F_n):=g\cdot\varphi(F_n).
$$
For $G=\Z$ and $F_n=[0,n]\cap\Z$, this compatibility condition is automatic. This suggests defining, for any $G$-invariant subset $\mathcal{S}$, the {\bf set of $\mathcal{S}$-maps} as
$$
\Maps(\mathcal{S}) := \{\varphi\colon \mathcal{S} \to V \mid \varphi \text{ is bounded and $G$-equivariant}\},
$$
where
$$
\vertiii{\varphi}_{\mathcal{S}} := \sup_{F \in \mathcal{S}}\left\|\frac{\varphi(F)}{|F|}\right\|,
$$
and a set map $\varphi$ is \emph{bounded} if $\vertiii{\varphi}_{\mathcal{S}} < \infty$. Similarly, we can extend the definitions of $\lim_{F \to G}$, $\limsup_{F \to G}$, etc., by restricting all of them to sets $F \in \mathcal{S}$. In particular, a set map $\varphi\colon \mathcal{S} \to V$ is \emph{asymptotically additive} if
$$
\inf_{v \in V} \limsup_{\substack{F \to G\\F \in \mathcal{S}}} \left\|\frac{\varphi(F)}{|F|} - A_Fv\right\| = 0.
$$
Similarly, in this context, when dealing with partitions $\Part$ of a set $F \in \mathcal{S}$, we will require that all elements $E \in \Part$ belong to $\mathcal{S}$ as well. Then, we say that a set map $\varphi\colon \mathcal{S} \to V$ is \emph{almost additive} if there exists $b \in \NMaps_\R$ such that
$$
\left\|\varphi(F)-\sum_{E \in \mathcal{P}} \varphi(E)\right\| \leq \sum_{E \in \mathcal{P}} b(E),
$$
for every $F \in \mathcal{S}$ and every monotone invariant partition $\Part \subseteq  \mathcal{S}$ of $F$. For example, if $G = \Z^d$ and $\mathcal{S}$ is the set of all $d$-dimensional boxes (or, equivalently, intervals for the case $d=1$), given $F \in \mathcal{S}$, we only consider partitions $\Part$ of $F$ such that every $E \in \Part$ is a $d$-dimensional box, where the monotonic invariance of $\Part$ is guaranteed.

\subsection{Almost additivity with constant error} \label{sec:aa-constant}

In the case $G=\Z$, it is common to deal with sequences $(f_n)_n$ of bounded continuous functions $f_n\colon X \to \R$, where $X$ is some topological space and $T\colon X \to X$ is a homeomorphism. If we consider $\mathcal{S} = \{[m,n) \cap \Z: m < n\}$, $V = \mathcal{C}_b(X)$, and $\koopman$ to be the Koopman representation $\koopman\colon \Z \to \Isom(\mathcal{C}_b(X))$ given by $\koopman(-n)(f)(x) =  f(T^n x)$, we can see that this fits in our context and that our formalism generalizes it. 

In this particular case, there have also been studied (see \cite{barreira2,barreira1, iommi-yuki, mummert}) sequences $(f_n)_n$ that satisfy the following property: there exists $C > 0$ such that for every $m,n \in \N$,
\begin{equation}
\label{eqn2}
\left|f_{n+m}(x) -  f_n(x) - (f_{m} \circ T^n)(x)\right| \leq C   \qquad   \text{for every } x \in X.
\end{equation}
By iterating Equation \eqref{eqn2}, we have that for any integers $M < N$,
\begin{equation}
\label{eqn3}
\left|(f_{N-M} \circ T^M)(x) - \sum_{i=0}^{k-1}(f_{n_{i+1} - n_i} \circ T^{n_i})(x)\right| \leq kC \qquad  \text{for every } x \in X,
\end{equation}
where $k \geq 1$ and $(n_i)_{i=0}^k$ are integers such that $M = n_0 < \cdots < n_i < n_{i+1} < \cdots < n_k = N$. Since $(\mathcal{C}_b(X),\|\cdot\|_\infty,\ll)$ is a Banach lattice endowed with a pointwise order, we can identify Equation \eqref{eqn3} with
\begin{equation}
\label{eqn4}
\left|\varphi(F) - \sum_{E \in \Part} \varphi(E)\right| \ll \sum_{E \in \Part} C\mathbbm{1},
\end{equation}
where $F = [M,N) \cap \Z$ and $\Part = \{[n_i,n_{i+1}) \cap \Z\}_{i=0}^{k-1}$. Since the Koopman representation $\koopman$ is unital and $\vertG{C\mathbbm{1}} = \lim_{n \to \infty} \frac{C}{n} = 0$, we conclude that $(f_n)_n$ can be identified with a set map that is Riesz-almost additive of constant type. Almost additivity with constant error also implies
asymptotic additivity in the sequential setting \cite[Proposition~2.1]{zhao2011asymptotically}. Therefore, the associated $\mathcal S$-map satisfies
all the weak forms of additivity discussed. Considering this, we will say that a sequence $(f_n)_n$ satisfying Equation \eqref{eqn2} is {\bf almost additive with constant error}.

\begin{example}
Let $X=\{0,1\}^\Z$ be the full $\Z$-shift and $T\colon X \to X$ the shift map. Given two positive matrices $A_0$ and $A_1$, $x \in \{0,1\}^\Z$, and $n \in \N$, we define
$$
f_n(x) = \log \|A_{x_0} \cdots A_{x_{n-1}}\|,
$$
where $\|\cdot\|$ is some sub-multiplicative matrix norm. In \cite[Lemma~2.1]{feng1}, it is proven that $(f_n)_n$ is almost additive with constant error, and hence Riesz-almost additive of constant type.
\end{example}

\begin{example}
Let $X=\{0,1\}^\Z$ be the full $\Z$-shift, $T\colon X \to X$ the shift map, and $\phi\colon X \to \R$ a continuous function. A \emph{weak Gibbs measure (for $\phi$)} is any Borel probability measure on $X$ such that there exists $P \in \R$ and a sequence of positive real numbers $(K_n)_n$ with $\frac{1}{n}\log K_n \to 0$ as $n \to \infty$ such that, for every $x \in X$ and $n \in \N$,
$$
K_n^{-1} \leq \frac{\mu([x_0 \cdots x_{n-1}])}
{\exp(-nP + S_n\phi(x))} \leq K_n,
$$
where $S_n \phi(x) = \sum_{i=0}^{n-1} \phi(T^ix)$ and $[x_0 \cdots x_{n-1}]$ denotes the cylinder set of the first $n$ coordinates induced by $x$. If we can choose $(K_n)_n$ to be constant, we say that $\mu$ is a \emph{Gibbs measure (for $\phi$)}.

Given a weak Gibbs measure $\mu$, consider the sequence of continuous functions $(f_n)_n$ defined by $f_n(x) = \log \mu([x_0 \cdots x_{n-1}])$. If $\mu$ is a Gibbs measure, then $(f_n)_n$ is almost additive with constant error. More generally, the weak Gibbs property gives
$$
\|f_n-S_n(\phi-P)\|_\infty\leq\log K_n=o(n).
$$
Thus $(f_n)_n$ is asymptotically additive, with additive
realization $\phi-P$. The extension argument in the remark following
Proposition~\ref{prop:erdos} then shows that it is
Riesz-almost additive of constant type. These notions are strictly more general than almost additivity with constant error. Indeed, the sequence
$g_n=\sqrt{n}\mathbbm{1}$ has additive realization $0$ and is Riesz-almost additive of constant type with error $\nmap(F)=\sqrt{|F|}\mathbbm{1}$.
However,
$\|g_{2n}-g_n-g_n\circ T^n\|_\infty
=(2-\sqrt{2})\sqrt{n}\to\infty$,
so it is not almost additive with constant error.
\end{example}

In \cite[Proposition~2.1]{zhao2011asymptotically}, it is proven that almost additive with constant error implies asymptotical additivity in the context of sequences of continuous functions. We stress that, in light of the previous discussion, Theorem~\ref{thm_zzc} is a far reaching generalization of this result. 

\subsection{Beyond the constant error case}

A more general property that retains the simplicity of the one given in Equation \eqref{eqn2} is to ask a sequence $(f_n)_n$ to satisfy
$$
\left\|f_{n+m} -  f_n - f_{m} \circ T^n\right\|_\infty \leq C_{n,m}
$$
for some double-indexed sequence $(C_{n,m})_{n,m}$. Then, a natural question is whether $(f_n)_n$ is still guaranteed to be (Riesz-)almost additive. Two common choices for $C_{n,m}$, at least  in the realm of sequences of real numbers (e.g., see \cite{de1952some, furedi2020nearly} for a related discussion about \emph{nearly sub-additive sequences}), are $C_{n+m}$ and $C_n + C_m$, where $(C_n)_n$ is a sequence of positive real numbers such that $\lim_{n \to \infty} \frac{C_n}{n} = 0$. The following proposition treats bounds of the form $C_n+C_m$ under additional monotonicity and summability assumptions. Under these assumptions, it also covers bounds of the form $C_{n+m}$, as the sequence $D_n=C_{2n}$ is
non-decreasing and satisfies
$$
C_{n+m}\leq D_n+D_m, \qquad \sum_{n\geq1}\frac{D_n}{n^2}
\leq 4\sum_{n\geq1}\frac{C_n}{n^2}<\infty.
$$

\begin{proposition}
\label{prop:erdos}
Let $X$ be a topological space and let $(f_n)_n$ be a sequence of bounded continuous real-valued functions in $\mathcal{C}_b(X)$ such that
$$
\left\|f_{n+m} -  f_n - f_{m} \circ T^n\right\|_\infty \leq C_n + C_m   \qquad   \text{for all } m,n \in \N,
$$
where $(C_n)_n$ satisfies that $0 \leq C_n \leq C_{n+1}$ and
$$
\sum_{n \geq 1}\frac{C_n}{n^2} < \infty.
$$
Then, $(f_n)_n$ is asymptotically additive.
\end{proposition}

\begin{proof}
It suffices to prove that for any $\epsilon > 0$, there exist $f \in \mathcal{C}_b(X)$ and $n_0 \in \N$ such that
$$
\frac{1}{n}\left\|f_n-S_n f\right\|_{\infty} \leq \epsilon \qquad \text{for every $n \geq n_0$}.
$$
Set $\tilde C_0=0$ and, for $n\geq1$, define
$$
\tilde C_n=5n\sum_{k\geq n}\frac{C_{5k}}{k^2}.
$$
For every $n\in\N$, let $f_n^+=f_n+\tilde C_n\mathbbm{1}$.
Notice that
\begin{equation}
\label{4-sub-ad}
f^+_{n+m} \ll f^+_n + f^+_m \circ T^n   \qquad \text{for all }m,n\geq 1 \text{ such that }\frac{1}{4}\leq\frac{n}{m}\leq 4.
\end{equation}
Indeed, following \cite[Theorem~23]{de1952some}, we have that
$$
f^+_{n+m} - f^+_n - f^+_m \circ T^n = (f_{n+m} - f_n - f_m \circ T^n) + (\tilde{C}_{n+m} - \tilde{C}_n - \tilde{C}_m)\mathbbm{1}  \ll 0,
$$
where we have used that $f_{n+m} - f_n - f_m \circ T^n \ll (C_n + C_m)\mathbbm{1} \ll (C_{5n} + C_{5m})\mathbbm{1}$ and
\begin{align*}
\tilde{C}_{n+m} - \tilde{C}_n - \tilde{C}_m &   \leq - 5n\sum_{n \leq k < n+m}{\frac{C_{5k}}{k^2}} - 5m\sum_{m \leq k < n+m}{\frac{C_{5k}}{k^2}} \\
&   \leq - 5nC_{5n}\left(\frac{1}{n}-\frac{1}{n+m}\right) - 5mC_{5m}\left(\frac{1}{m}-\frac{1}{n+m}\right) \\
&   \leq -C_{5n} - C_{5m}, 
\end{align*}
considering that $\frac{1}{4}\leq n/m\leq4$ and $\sum_{r \leq k < s}\frac{1}{k^2} > \frac{1}{r} - \frac{1}{s}$ for integers $1 \leq r < s$. Similarly, the sequence $(f^-_n)_n$ given by $f^-_n = f_n - \tilde{C}_n\mathbbm{1}$ satisfies that
\begin{equation}
\label{4-sup-ad}
f^-_n + f^-_m \circ T^n \ll f^-_{n+m}  \qquad \text{for all }m,n\geq 1
\text{ such that }\frac{1}{4}\leq\frac{n}{m}\leq 4.
\end{equation}

Observe that $\frac{\tilde{C}_n}{n} \to 0$ as $n \to \infty$. Fix a positive integer $m$ such that $\frac{\tilde{C}_n}{n} < \frac{\epsilon}{3}$ for every $n \geq m$, and let
$$
K = \max_{0 \leq j < 3m} \max\{\|f^+_j\|_\infty, \|f^-_j\|_\infty\}.
$$

Take $n_0 > 4m$ such that $\frac{6K}{n_0} < \frac{\epsilon}{3}$. Notice that $\frac{\tilde{C}_n}{n}< \frac{\epsilon}{3}$ for every $n \geq n_0$. For each $n \geq n_0$, write $n=ms+\ell$ for $0 \leq \ell < m$. Then, for $0 \leq j< m$, we have that $n = (m+j) + (s - 3)m + (2m + \ell - j)$ with $m \leq m+j < 2m$, $m < 2m+\ell-j < 3m$, and $s \geq 4$. Considering this, it is possible to write
\begin{align*}
n     = \left((m+j) + \left\lceil \frac{s-3}{2} \right\rceil m\right) + \left(\left\lfloor \frac{s-3}{2} \right\rfloor m + (2m + \ell - j)\right).
\end{align*}
Notice that the quotient between the two terms is bounded above by $4$ and below by $\frac{1}{4}$. We subdivide only terms containing at least two prescribed blocks, stopping at individual blocks. For interior terms $am$ with $a\geq2$, we use $am=\lfloor a/2\rfloor m+\lceil a/2\rceil m$.

In addition, the first term can be partitioned as
$$
\left((m+j)+
\left\lfloor\frac{1}{2}\left\lceil\frac{s-3}{2}\right\rceil\right\rfloor m\right) + \left(\left\lceil\frac{1}{2}
\left\lceil\frac{s-3}{2}\right\rceil\right\rceil m\right).
$$
The second term, whenever it is to be subdivided, is partitioned as
$$
\left(\left\lceil\frac{1}{2}\left\lfloor\frac{s-3}{2}\right\rfloor\right\rceil m\right) + \left(\left\lfloor\frac{1}{2}
\left\lfloor\frac{s-3}{2}\right\rfloor\right\rfloor m +(2m+\ell-j)\right).
$$
Again, the quotient between two terms obtained from one is bounded above by $4$ and below by $\frac{1}{4}$. Iterating this procedure with each term and ensuring that the leftmost term has the form $(m+j)+am$, the rightmost term has the form $am+(2m+\ell-j)$, and all the other terms have the form $am$, for suitable non-negative integers $a$, it is possible to write
$$
n=(m+j)+
\underbrace{m+\cdots+m}_{(s-3)\text{ times}}
+(2m+\ell-j)
$$
in such a way that in every step the quotient between every two terms obtained from a single one is bounded from above by $4$ and below by $\frac{1}{4}$. Therefore, due to Equation \eqref{4-sub-ad} and considering the decomposition of $n$ described above, we have that
$$
f^+_n \ll f^+_{m+j}
+\sum_{r=0}^{s-4}f^+_m\circ T^{(r+1)m+j}
+f^+_{2m+\ell-j}\circ T^{(s-2)m+j}.
$$
Next, summing $j$ from $0$ to $m-1$, we have
$$
mf^+_n \ll \sum_{j=0}^{m-1} f^+_{m+j} + \sum_{i=0}^{(s-3) m-1} f^+_m \circ T^{m+i} + \sum_{j=0}^{m-1} f^+_{2m+\ell-j} \circ T^{(s-2) m + j}.
$$
Similarly, using Equation \eqref{4-sup-ad}, we have
$$
\sum_{j=0}^{m-1} f^-_{m+j} + \sum_{i=0}^{(s-3)m-1} f^-_m \circ T^{m+i} + \sum_{j=0}^{m-1} f^-_{2m+\ell-j} \circ T^{(s-2)m + j} \ll m f^-_n.
$$

After division by $m$, the two endpoint sums contribute
at most $2K$ in norm, and the $3m+\ell<4m$ missing
Birkhoff terms contribute at most $4K$. Thus,
$$
f^+_n \ll 6K\mathbbm{1}+\sum_{i=0}^{n-1} \frac{1}{m} f^+_m \circ T^i \qquad \text{and} \qquad -6K\mathbbm{1}+\sum_{i=0}^{n-1} \frac{1}{m} f^-_m \circ T^i \ll f^-_n.
$$

Hence,
$$
\left(-6K - \frac{n}{m}\tilde{C}_m + \tilde{C}_n\right)\mathbbm{1} \ll f_n-S_n \frac{f_m}{m} \ll \left(6K + \frac{n}{m}\tilde{C}_m - \tilde{C}_n\right)\mathbbm{1}.
$$

Thus, if $f = \frac{f_m}{m}$, we have
$$
\frac{1}{n}\left\|f_n-S_n f\right\|_\infty  \leq \frac{6K}{n} + \frac{\tilde{C}_n}{n} + \frac{\tilde{C}_m}{m}  <   \frac{\epsilon}{3} + \frac{\epsilon}{3} + \frac{\epsilon}{3} = \epsilon,
$$
and the result follows.
\end{proof}

\begin{remark}
Proposition~\ref{prop:erdos} combined with Proposition~\ref{lem:asymp-almost} and Proposition~\ref{prop:riesz-almost-constant} guarantees that every sequence $(f_n)_n$ satisfying the hypotheses of Proposition~\ref{prop:erdos} is Riesz-almost additive (of constant type), and thus almost additive. Indeed, by \cite[Theorem~2.1]{cuneo2020additive}, there exists $f\in\mathcal C_b(X)$ such that $\|f_n-S_nf\|_\infty/n\to0$. Extending the error $\varphi(F)-S_Ff$ by zero outside $\mathcal S$ gives $\nmap\in\NMaps$, so $F\mapsto S_Ff+\nmap(F)$ is an asymptotically additive extension to $\FSet$, to which the latter two propositions apply before restricting back to intervals. Notice also that if $(C_n)_n$ is constant in Proposition~\ref{prop:erdos}, the same conclusion holds, generalizing \cite[Proposition~2.1]{zhao2011asymptotically}.
\end{remark}

\begin{question}
If $(C_n)_n$ is a non-decreasing sequence of positive real numbers such that $\sum_{n \geq 1} \frac{C_n}{n^2} < \infty$, then $\lim_{n \to \infty} \frac{C_n}{n} = 0$. Is it possible to weaken the assumptions on $(C_n)_n$ in Proposition~\ref{prop:erdos} to just $0 \leq C_n \leq C_{n+1}$ and $\lim_{n \to \infty} \frac{C_n}{n} = 0$? An analogue question in the setting of nearly sub-additive sequences of real numbers is answered negatively in \cite{furedi2020nearly}.
\end{question}

\subsection*{Acknowledgments}

The authors thank the referee for the careful reading and helpful corrections, which substantially improved the manuscript.

\subsection*{Funding}

R.B. was partially supported by ANID/FONDECYT Regular 1240508. G.I. was partially supported by ANID/FONDECYT Regular 1230100. Both authors were partially supported by Avanza UC AV25085.

\bibliographystyle{abbrv}
\bibliography{references}

\end{document}